\documentclass[11pt]{article}

\usepackage[margin=1in]{geometry}
\usepackage[T1]{fontenc}
\usepackage[utf8]{inputenc}
\usepackage{amsmath,amssymb,amsthm,mathtools,bm}
\usepackage{enumitem}
\usepackage{microtype}
\usepackage{xcolor}
\usepackage[round,authoryear]{natbib}
\usepackage[colorlinks=true,linkcolor=blue,citecolor=blue,urlcolor=blue,pagebackref=true]{hyperref}
\renewcommand*{\backrefalt}[4]{%
  \ifcase #1\relax
  \or (cited on p.~#2)%
  \else (cited on pp.~#2)%
  \fi}
\numberwithin{equation}{section}
\mathtoolsset{showonlyrefs=true}

\newtheorem{theorem}{Theorem}[section]
\newtheorem{proposition}[theorem]{Proposition}
\newtheorem{lemma}[theorem]{Lemma}
\newtheorem{corollary}[theorem]{Corollary}
\newtheorem{condition}[theorem]{Condition}
\theoremstyle{definition}

\theoremstyle{remark}

\title{On the minimax-rate optimality of approximate Bayesian computation in nonparametric problems}
\author{Hien Duy Nguyen\\
\small School of Computing, Engineering and Mathematical Sciences, La Trobe University,\\
\small Bundoora 3086, Victoria, Australia\\
\small Institute of Mathematics for Industry, Kyushu University, Nishi Ward, Fukuoka 819-0395, Japan\\
\small \texttt{h.nguyen5@latrobe.edu.au}}
\date{}

\begin{document}
\maketitle

\begin{abstract}
Approximate Bayesian computation (ABC) replaces likelihood evaluation by simulation and comparison of observed and synthetic data. We establish minimax-rate guarantees for nonparametric ABC under random-series priors with simulable finite-dimensional coordinates. The contraction theorem uses local prior mass, bounds on ABC acceptance probabilities, and control of prior mass outside a sieve. In fixed-design orthogonal-series regression with centered $g$-and-$k$ errors, an infinite Gaussian series prior with a compact scale hyperprior yields minimax-rate contraction and a minimax-rate clipped posterior mean. In compound Poisson decompounding, only random sums are observed and the target is the underlying jump density. With an unknown count intensity in a fixed compact subinterval of $(0,\pi/2)$, we prove stability of the zero-count-augmented trigonometric population summaries and use a square-root Gaussian series prior on the space of probability density functions. Over bounded periodic Sobolev classes of smoothness $\alpha>d/2$, a polynomially enlarged synthetic sample yields ABC contraction at rate $n^{-\alpha/(2\alpha+d)}$ and posterior mean squared risk of order $n^{-2\alpha/(2\alpha+d)}$, matching a lower bound for the aggregate-observation model. Rejection-ABC Monte Carlo approximations inherit these rates under sufficient sampling budgets.
\end{abstract}


\section{Introduction}
\label{sec:introduction}

\subsection{Approximate Bayesian computation}

Approximate Bayesian computation (ABC) is a likelihood-free approach for Bayesian inference in models for which simulating data is feasible but evaluating the likelihood is impossible or computationally prohibitive. In its basic rejection form, one draws a parameter $\theta$ from a prior measure, simulates synthetic data $Z$ from the model defined by $\theta$, computes a discrepancy between $Z$ and the observed data $X$, and retains $\theta$ if the discrepancy is below a tolerance. The method has its origins in population genetics and coalescent models \citep{Tavare1997,Pritchard1999}, with regression-adjusted ABC introduced by \citet{BeaumontZhangBalding2002}. Broader accounts of rejection, importance-sampling, Markov chain Monte Carlo, sequential Monte Carlo, regression-adjusted, synthetic-likelihood, and discrepancy-based variants are given by \citet{MarinPudloRobertRyder2012}, \citet{KarabatsosLeisen2018}, and \citet{SissonFanBeaumont2018}.

Let $S_n^X:\mathcal X_n\to\mathcal S_n$ and $S_n^Z:\mathcal Z_n\to\mathcal S_n$ be measurable observed and synthetic summary maps, and let $\rho_n$ be a measurable discrepancy on $\mathcal S_n$. With the indicator kernel and tolerance $\tau_n$, the ideal ABC posterior is
\begin{equation}
  \Pi_n^{\mathrm{ABC}}(\mathrm d\theta\mid X)
  \propto
  \Pi_n(\mathrm d\theta)\,\mathrm{Q}_\theta^{(n)}\left\{\rho_n(S_n^Z(Z),S_n^X(X))\le \tau_n\right\},
  \label{eq:intro-abc-posterior}
\end{equation}
where $\Pi_n$ is the prior and $\mathrm Q_\theta^{(n)}$ is the simulator law; the probability is over an independent draw $Z$, with $X$ held fixed. The synthetic sample size may differ from the observed sample size. The empirical rejection-ABC posterior assigns equal mass to the accepted parameter draws and uses a fallback measure when there are no acceptances.

\subsection{Nonparametric problems and minimax rates}

We study two illustrative nonparametric problems. The first is fixed-design orthogonal-series regression with non-Gaussian errors. We use a centered $k=0$ subfamily of the quantile-defined $g$-and-$k$ distribution introduced by \citet{HaynesMacGillivrayMengersen1997} and studied computationally by \citet{RaynerMacGillivray2002}; related likelihood-free $g$-and-$k$ applications appear in \citet{BerntonJacobGerberRobert2019} and \citet{NguyenArbelLuForbes2020}. The regression simulator evaluates the quantile map directly, whereas likelihood evaluation requires numerical inversion of that map. Orthogonal coefficient summaries permit estimation of the regression sequence over Sobolev ellipsoids at squared $\ell^2$-risk rate $n^{-2\alpha/(2\alpha+1)}$.

The second problem is compound Poisson decompounding: an unknown jump density on $[0,1]^d$ must be recovered from totals of unobserved Poisson numbers of jumps; see \citet[Section~1.2]{CekanaviciusNovak2025} for the compound Poisson construction. Nonparametric decompounding is an established inverse problem; kernel estimation is studied by \citet{VanEsGugushviliSpreij2007}, and Bayesian analyses are given by \citet{GugushviliVanderMeulenSpreij2015,GugushviliVanderMeulenSpreij2018}. ABC for fitting insurance loss models from aggregate data is developed by \citet{GoffardLaub2021}. We consider a bounded periodic Sobolev class of jump densities, with known smoothness $\alpha>d/2$ and an unknown Poisson intensity $\beta_0$ in a fixed compact subinterval of $(0,\pi/2)$. The observed-data likelihood involves a mixture of convolution densities, while synthetic totals are generated from the count and jump distributions. We establish the squared $\mathcal L^2$-risk rate $n^{-2\alpha/(2\alpha+d)}$ for estimating the latent jump density, uniformly over the intensity interval, with a matching lower bound based on the aggregate observations.

\subsection{Main contributions}

The paper makes three contributions. First, it gives a generic ABC contraction theorem that bounds ABC acceptance probabilities for parameters away from the truth on a measurable subset $\mathcal G_n$ of the parameter space, called a prior sieve, whose complement has sufficiently small prior probability. Under these conditions, the prior need not be conditioned to lie in the target class; it may also be fixed, provided it satisfies a prior-sieve bound. The required local prior mass is bounded below by $\exp\{-C K_n\}$ at the target radius $\epsilon_n$, where $K_n$ is the effective summary dimension. The argument is analogous to the prior-mass and sieve balance in nonparametric posterior-contraction theory; compare \citet[Chapter 8]{GhosalVanderVaart2017} and \citet[Chapters 1--3]{Castillo2024}. Random-series priors and adaptation under ordinary likelihoods are treated in \citet{GhosalLemberVanderVaart2003} and \citet{ShenGhosal2015}.

Second, for orthogonal-series regression we use the fixed infinite series
\begin{equation}
  \theta_j=\Lambda j^{-\alpha-1/2}G_j,\qquad j\ge1,
\end{equation}
where the $G_j$ are independent standard Gaussian variables and $\Lambda$ has a compactly supported hyperprior. Here, the prior does not depend on $n$ or on the Sobolev radius $R$ defining the target class. On exponentially likely tail sieves, sequences far from the truth have correspondingly distant summary centers. For centered $g$-and-$k$ errors, the ideal ABC posterior contracts at the minimax rate, a clipped posterior mean attains the corresponding squared-risk rate, and a matching lower bound follows from a local shift-Kullback--Leibler inequality.

Third, for decompounding we combine trigonometric summaries and the proportion of zero-count periods with a square-root series prior for the jump density and an independent prior for the unknown count intensity. If $g$ is a nonzero frequency-$J$ trigonometric polynomial, then
\begin{equation}
  p_g=\frac{g^2}{\lVert g\rVert_{\mathcal L^2}^2}
\end{equation}
is a density with frequencies at most $2J$. Our use of square-root density coordinates is related to the computational constructions of \citet{HolbrookLanStreetsShahbaba2020}; an alternative route based on exponential-series density priors is studied by \citet{Scricciolo2006}. We prove a uniform inverse bound relating the augmented aggregate-summary centers to the low-frequency jump-density coefficients when the intensity is unknown. Exact jump simulation uses a deterministic rejection envelope, and compound Poisson simulation then produces the totals. Together with a polynomially enlarged synthetic sample, these properties give minimax-rate ABC contraction and posterior mean risk for the jump density. The lower bound restricts the model to one intensity value and uses the fixed-intensity Kullback--Leibler comparison in \citet[Lemma 1]{GugushviliVanderMeulenSpreij2015}.

\subsection{Relation to the ABC literature}
\label{sec:abc-literature}

ABC asymptotic theory has primarily studied posterior concentration, limiting shape, and Monte Carlo efficiency in fixed-dimensional settings. \citet{FrazierMartinRobertRousseau2018} give general concentration and limiting-shape results for summary-based ABC. \citet{LiFearnhead2018a} show that regression adjustment can recover calibrated uncertainty under suitable bandwidth choices, and \citet{LiFearnhead2018b} analyze Monte Carlo efficiency. Appendix~\ref{app:parametric} gives a fixed-dimensional minimax benchmark.

A second direction studies full-data or discrepancy-based ABC. Wasserstein, Kullback--Leibler, energy, surrogate-posterior, and integral-probability-semimetric approaches are developed by \citet{BerntonJacobGerberRobert2019}, \citet{JiangWuWong2018}, \citet{NguyenArbelLuForbes2020}, \citet{ForbesNguyenNguyenArbel2022}, and \citet{LegramantiDuranteAlquier2025}. \citet{NguyenNguyenArbelForbes2025} establish concentration under weak discrepancy assumptions. Here, we ask when summary-based ABC can reproduce the classical minimax rates of nonparametric estimation.

A third direction concerns high-dimensional summaries. \citet{Blum2010}, \citet{BarberVossWebster2015}, and \citet{NottOngFanSisson2018} show that rejection or kernel ABC deteriorates with summary dimension. Here $K_n\asymp n\epsilon_n^2$ diverges, and the normalizer lower bound gives the sufficient prior-rejection condition $N_n\exp\{-\kappa K_n\}\to\infty$. This is a sufficient transfer guarantee but does not provide a computational lower bound. MCMC-ABC \citep{MarjoramMolitorPlagnolTavare2003}, SMC-ABC \citep{BeaumontCornuetMarinRobert2009,DelMoralDoucetJasra2012}, regression-adjusted ABC \citep{BeaumontZhangBalding2002,LiFearnhead2018a}, and synthetic-likelihood algorithms \citep{Wood2010,FrazierNottDrovandiKohn2022} may reduce the simulation cost but require separate sampler-specific analyses.

\subsection{Manuscript outline}

Section~\ref{sec:notation} defines the ABC posterior objects, and Section~\ref{sec:meta} gives the contraction and Monte Carlo transfer results. Section~\ref{sec:gk-regression} treats $g$-and-$k$ error regression, Section~\ref{sec:density} treats nonparametric decompounding, and Section~\ref{sec:discussion} discusses limitations and extensions. The appendices collect auxiliary technical results and proofs. Appendix~\ref{app:parametric} gives minimax-rate results in the parametric setting.

\section{Technical preliminaries}
\label{sec:notation}

Let $(\Omega,\mathfrak F,\mathrm P)$ be a sufficiently rich probability space supporting all observed data, simulator outputs, prior samples, and Monte Carlo randomness. Write $\mathrm E$ for expectation under $\mathrm P$. Subscripts specify the probability law or conditioning and are omitted when it is clear from context.  For vectors in $\mathbb R^K$, $\left\lVert\cdot\right\rVert_2$ denotes the Euclidean norm and $\mathbb S^{K-1}=\{u\in\mathbb R^K:\left\lVert u\right\rVert_2=1\}$ denotes the unit sphere.  For square-summable sequences, $\left\lVert\theta\right\rVert_{\ell^2}^2=\sum_{j=1}^{\infty}\theta_j^2$.  For a sequence $x=(x_1,x_2,\ldots)$ and $K\ge1$, write $\bm x_K=(x_1,\ldots,x_K)$ and $\bm x_{>K}=(x_{K+1},x_{K+2},\ldots)$.  For functions on a measurable set $\mathcal X$ with Lebesgue measure, $\left\lVert f\right\rVert_{\mathcal L^2}^2=\int_{\mathcal X} f(x)^2\,\mathrm dx$; when the argument is a function and no ambiguity can arise, we also write $\lVert f\rVert_2=\lVert f\rVert_{\mathcal L^2}$.  We write $a_n\lesssim b_n$ when $a_n\le Cb_n$ for a constant $C$ independent of $n$, $a_n\gtrsim b_n$ when $b_n\lesssim a_n$, and $a_n\asymp b_n$ when both inequalities hold.  We use $o(\cdot)$ and $O(\cdot)$ in their usual Landau sense.

For probability measures $\mathrm P,\mathrm Q$ on a common measurable space $(\mathcal Y,\mathfrak Y)$, define
\begin{equation}
  \mathrm{TV}(\mathrm P,\mathrm Q)=\sup_{A\in\mathfrak Y}|\mathrm P(A)-\mathrm Q(A)|,
  \qquad
  \mathrm{KL}(\mathrm P,\mathrm Q)=\int_{\mathcal Y}\log\left(\frac{\mathrm d\mathrm P}{\mathrm d\mathrm Q}(y)\right)\mathrm P(\mathrm dy)
  \label{eq:divergence-definitions}
\end{equation}
when $\mathrm P\ll \mathrm Q$, and set $\mathrm{KL}(\mathrm P,\mathrm Q)=\infty$ otherwise. When $\mathrm P$ and $\mathrm Q$ admit densities $p$ and $q$ with respect to a common dominating measure, we also write $\mathrm{TV}(p,q)$ and $\mathrm{KL}(p,q)$.

For a parameter $g$ (written as $f_0$, $\theta_0$, or $(\beta_0,p_0)$ according to the model), write $\mathrm P_g^{(n)}$ and $\mathrm E_g^{(n)}$ for probability and expectation under the observed model indexed by $n$. Write $\mathrm P_g$ and $\mathrm E_g$ for the corresponding probability and expectation for one observation.

Let $\Theta$ be a separable metric parameter space with metric $d$, and let $\mathfrak B(\Theta)$ denote its Borel $\sigma$-field.  For each $n$, let $(\mathcal X_n,\mathfrak X_n)$ and $(\mathcal Z_n,\mathfrak Z_n)$ be the observed and synthetic data spaces.  Let $\{\mathrm P_\theta^{(n)}:\theta\in\Theta\}$ be a probability kernel from $(\Theta,\mathfrak B(\Theta))$ to $(\mathcal X_n,\mathfrak X_n)$, and let $\{\mathrm Q_\theta^{(n)}:\theta\in\Theta\}$ be a simulator kernel from $(\Theta,\mathfrak B(\Theta))$ to $(\mathcal Z_n,\mathfrak Z_n)$.  Let $X^{(n)}$ denote the observed data and, for each $\theta\in\Theta$, let $Z^{(n)}\sim\mathrm Q_\theta^{(n)}$ denote an independent synthetic dataset.  In the usual matched-sample setting, $\mathrm Q_\theta^{(n)}=\mathrm P_\theta^{(n)}$; allowing a separate simulator kernel also covers the enlarged synthetic sample used in Section~\ref{sec:density}.  The simulator is assumed available even when its likelihood density is unavailable.  Let
\begin{equation}
  S_n^X:(\mathcal X_n,\mathfrak X_n)\to(\mathcal S_n,\mathfrak S_n),
  \qquad
  S_n^Z:(\mathcal Z_n,\mathfrak Z_n)\to(\mathcal S_n,\mathfrak S_n)
\end{equation}
be measurable observed and synthetic summary maps, where $(\mathcal S_n,\mathfrak S_n)$ is equipped with a jointly measurable metric or semimetric $\rho_n$, typically induced by a norm.  The two maps may coincide when the observed and synthetic data have the same structure.  For a tolerance $\tau_n>0$, the indicator-kernel ABC likelihood is
\begin{equation}
  L_n^{\mathrm{ABC}}(\theta;X)
  =\mathrm Q_\theta^{(n)}\left\{\rho_n(S_n^Z(Z^{(n)}),S_n^X(X^{(n)}))\le \tau_n\right\}.
  \label{eq:abc-likelihood-general}
\end{equation}
Given a prior $\Pi_n$ on $\Theta$, define
\begin{equation}
  D_n(X)=\int_\Theta L_n^{\mathrm{ABC}}(\theta;X)\Pi_n(\mathrm d\theta).
\end{equation}
On the event $D_n(X)>0$, the ideal ABC posterior is
\begin{equation}
  \Pi_n^{\mathrm{ABC}}(B\mid X)
  =\frac{\int_B L_n^{\mathrm{ABC}}(\theta;X)\Pi_n(\mathrm d\theta)}{D_n(X)}.
  \label{eq:abc-posterior-general}
\end{equation}
Fix a deterministic reference value $\theta^\dagger\in\Theta$ and set $\Pi_n^{\mathrm{ABC}}(\cdot\mid X)=\delta_{\theta^\dagger}$ when $D_n(X)=0$.  We call $\delta_{\theta^\dagger}$ the fallback measure.  In \eqref{eq:abc-likelihood-general}, the probability is taken only over the independent simulator draw $Z^{(n)}$, with the observed value $X$ fixed.  Conditional on $X$, $D_n(X)$ is the probability that one parameter drawn from $\Pi_n$ and one synthetic dataset generated from that parameter are accepted.

For rejection ABC, generate independent prior samples $\theta_i\sim\Pi_n$ and synthetic data $Z_i^{(n)}\sim \mathrm Q_{\theta_i}^{(n)}$, $i=1,\ldots,N_n$, conditionally independently given $X$.  For notational economy, write $N=N_n$ in the subscripts below.  Let
\begin{equation}
  I_i=\mathbf{1}\left\{\rho_n(S_n^Z(Z_i^{(n)}),S_n^X(X^{(n)}))\le \tau_n\right\},
  \qquad A_{n,N}=\sum_{i=1}^{N_n}I_i.
  \label{eq:mc-objects-general}
\end{equation}
Conditional on the observed data $X$, write $\mathrm{P}_{\mathrm{MC}}(\cdot\mid X)$ and $\mathrm{E}_{\mathrm{MC}}(\cdot\mid X)$ for probability and expectation under the joint law of the prior draws, simulator draws, and indicators $\{(\theta_i,Z_i^{(n)},I_i):1\le i\le N_n\}$ generated above.  When either operator is nested inside an outer expectation over the observed data, the explicit conditioning argument on $X$ is suppressed when no ambiguity can arise.  When $A_{n,N}\ge1$, the empirical ABC posterior is
\begin{equation}
  \widehat \Pi_{n,N}^{\mathrm{ABC}}(B\mid X)=\frac{\sum_{i=1}^{N_n}I_i\mathbf{1}\{\theta_i\in B\}}{A_{n,N}}.
  \label{eq:empirical-posterior-general}
\end{equation}
When $A_{n,N}=0$, we set $\widehat\Pi_{n,N}^{\mathrm{ABC}}(\cdot\mid X)=\delta_{\theta^\dagger}$.

Under the minimax framework of \citet[Section 6.3]{GineNickl2016}, let $(\mathcal F_n)$ be a sequence of model or parameter classes with metric $d$. A sequence $\epsilon_n\to0$ is a minimax rate in the metric $d$ over $(\mathcal F_n)$ if
\begin{equation}
  \inf_{\widehat f_n}\sup_{f_0\in\mathcal F_n} \mathrm{E}_{f_0}^{(n)}d^2(\widehat f_n,f_0)\asymp \epsilon_n^2,
\end{equation}
where the infimum is over all measurable estimators.  Under the contraction convention adopted here, an ideal ABC posterior is minimax-rate optimal if there exists $M_0<\infty$ such that, for every fixed $M\ge M_0$,
\begin{equation}
  \sup_{f_0\in\mathcal F_n}\mathrm{E}_{f_0}^{(n)}\Pi_n^{\mathrm{ABC}}\left\{f:d(f,f_0)>M\epsilon_n\mid X\right\}\to0.
\end{equation}
If $\widehat f_n^{\mathrm{ABC}}=\int_{\Theta} f\,\Pi_n^{\mathrm{ABC}}(\mathrm df\mid X)$ is the ABC posterior mean, then it is minimax-rate optimal when
\begin{equation}
  \sup_{f_0\in\mathcal F_n}\mathrm{E}_{f_0}^{(n)}d^2(\widehat f_n^{\mathrm{ABC}},f_0)=O(\epsilon_n^2).
\end{equation}
Whenever a posterior mean is an element of an infinite-dimensional Hilbert space, the integral is interpreted as the Bochner integral of the identity map under the convention of \citet[Section 2.1]{LedouxTalagrand1991}.

\section{Abstract ABC contraction and Monte Carlo transfer}
\label{sec:meta}

Let $K_n\to\infty$ be an effective dimension and let $\epsilon_n\to0$ be a target rate in the metric $d$. In the applications below, $K_n\asymp n\epsilon_n^2$.

\begin{condition}[Abstract local-mass and prior-sieve conditions]
\label{ass:meta}
Let $\mathcal F_n\subset\Theta$ be the target class. There exist constants $C_\Pi,C_L,C_D,C_S,C_G,c>0$ and $M_\star<\infty$, events $\mathcal E_n=\mathcal E_n(f_0)$, measurable local sets $\mathcal B_n(f_0)\subset\Theta$, measurable prior sieves $\mathcal G_n\subset\Theta$, and prior measures $\Pi_n$ such that, for all sufficiently large $n$, uniformly over $f_0\in\mathcal F_n$:
\begin{enumerate}[label=(B\arabic*)]
\item $\mathrm P_{f_0}^{(n)}(\mathcal E_n^c)\le \exp\{-cK_n\}$.
\item $\Pi_n\{\mathcal B_n(f_0)\}\ge \exp\{-C_\Pi K_n\}$.
\item For every $f\in\mathcal B_n(f_0)$, $d(f,f_0)\le C_L\epsilon_n$ and, on $\mathcal E_n$,
\begin{equation}
  L_n^{\mathrm{ABC}}(f;X)\ge \exp\{-C_LK_n\}.
\end{equation}
\item For every fixed $M\ge M_\star$, there is $n_0(M)<\infty$ such that, for all $n\ge n_0(M)$, on $\mathcal E_n$,
\begin{equation}
  \sup_{\substack{f\in\mathcal G_n:\ d(f,f_0)>M\epsilon_n}}
  L_n^{\mathrm{ABC}}(f;X)
  \le \exp\{-C_D(M-C_S)^2K_n\}.
  \label{eq:meta-sieve-separation}
\end{equation}
\item $\Pi_n(\mathcal G_n^c)\le\exp\{-(C_\Pi+C_L+C_G)K_n\}$.
\end{enumerate}
\end{condition}

Condition (B1) controls $\mathrm P_{f_0}^{(n)}(\mathcal E_n^c)$, and (B2)--(B3) give a lower bound for the ABC normalizer $D_n(X)$ on $\mathcal E_n$. Condition (B4) bounds ABC acceptance probabilities for parameters farther than $M\epsilon_n$ from $f_0$ on an exponentially likely prior sieve, while condition (B5) controls the remaining prior mass using $0\le L_n^{\mathrm{ABC}}\le1$. If $\mathcal G_n=\operatorname{supp}(\Pi_n)$, then (B5) holds.

In normed-summary applications, let $s_n(f)$ denote the deterministic summary center and suppose $\tau_n\le A\epsilon_n$. On an event where $\rho_n\{S_n^X(X),s_n(f_0)\}\le C_X\epsilon_n$, acceptance implies
\begin{equation}
\begin{aligned}
  \rho_n\{S_n^Z(Z),s_n(f)\}
  &\ge \rho_n\{s_n(f),s_n(f_0)\}
      -\rho_n\{S_n^Z(Z),S_n^X(X)\}
      -\rho_n\{S_n^X(X),s_n(f_0)\}.
\end{aligned}
\label{eq:meta-separation-explanation}
\end{equation}
Thus, when the center of a sieve element is far from the true center, acceptance is possible only if the synthetic summary makes a comparably large deviation from its own center.

\begin{theorem}[Ideal ABC contraction]
\label{thm:meta-ideal}
Under Condition~\ref{ass:meta}, there exists $M_0<\infty$ such that, for every fixed $M\ge M_0$, there is $\gamma_M>0$ for which, on $\mathcal E_n$ and for all sufficiently large $n$,
\begin{equation}
  \Pi_n^{\mathrm{ABC}}\left\{f:d(f,f_0)>M\epsilon_n\mid X\right\}
  \le 2e^{-\gamma_MK_n}.
  \label{eq:meta-ideal-on-event}
\end{equation}
Consequently,
\begin{equation}
  \sup_{f_0\in\mathcal F_n}\mathrm E_{f_0}^{(n)}
  \Pi_n^{\mathrm{ABC}}\left\{f:d(f,f_0)>M\epsilon_n\mid X\right\}\to0.
  \label{eq:meta-ideal-contraction}
\end{equation}
\end{theorem}

\begin{proof}
See Appendix~\ref{app:meta-proofs}.
\end{proof}

\begin{corollary}[Normalizer lower bound]
\label{cor:meta-denom}
Under Condition~\ref{ass:meta}, there exist $\kappa<\infty$ and $c'\in(0,c)$ such that
\begin{equation}
  \sup_{f_0\in\mathcal F_n}\mathrm P_{f_0}^{(n)}
  \left\{D_n(X)<e^{-\kappa K_n}\right\}
  \le e^{-c'K_n}
  \label{eq:meta-denom-bound}
\end{equation}
for all sufficiently large $n$.
\end{corollary}

\begin{proof}
See Appendix~\ref{app:meta-proofs}.
\end{proof}

\begin{theorem}[Monte Carlo transfer for posterior mass]
\label{thm:meta-mc}
Assume Condition~\ref{ass:meta}, let $\kappa$ be as in Corollary~\ref{cor:meta-denom}, and let $M_0$ be as in Theorem~\ref{thm:meta-ideal}. If
\begin{equation}
  N_n e^{-\kappa K_n}\to\infty,
  \label{eq:meta-N-contract}
\end{equation}
then, for every fixed $M\ge M_0$,
\begin{equation}
  \sup_{f_0\in\mathcal F_n}\mathrm E_{f_0}^{(n)}\mathrm E_{\mathrm{MC}}
  \widehat\Pi_{n,N}^{\mathrm{ABC}}
  \left\{f:d(f,f_0)>M\epsilon_n\mid X\right\}\to0.
  \label{eq:meta-mc-contract}
\end{equation}
\end{theorem}

\begin{proof}
See Appendix~\ref{app:meta-proofs}.
\end{proof}

\begin{corollary}[Monte Carlo transfer for posterior means]
\label{cor:meta-mc-mean}
Assume the conditions of Theorem~\ref{thm:meta-mc}. Suppose $\Theta$ is contained in a separable Hilbert space, $d$ is the ambient norm metric (also used for means outside $\Theta$), and
\begin{equation}
  \sup_{n\ge1}\sup_{f_0\in\mathcal F_n}d(\theta^\dagger,f_0)<\infty.
  \label{eq:meta-fallback-bounded}
\end{equation}
Define the empirical ABC posterior mean by
\begin{equation}
  \widehat f_{n,N}^{\mathrm{ABC}}
  =\int_\Theta f\,\widehat\Pi_{n,N}^{\mathrm{ABC}}(\mathrm df\mid X).
  \label{eq:meta-empirical-mean}
\end{equation}
Suppose also that
\begin{equation}
  \sup_{f_0\in\mathcal F_n}\mathrm E_{f_0}^{(n)}
  \int_\Theta d^2(f,f_0)\Pi_n^{\mathrm{ABC}}(\mathrm df\mid X)
  \le C\epsilon_n^2,
  \label{eq:meta-ideal-second}
\end{equation}
that $e^{-c'K_n}=O(\epsilon_n^2)$, and that
\begin{equation}
  N_n e^{-\kappa K_n}\ge C_N\log(1/\epsilon_n)
  \label{eq:meta-N-risk}
\end{equation}
for $C_N$ sufficiently large. Then the empirical ABC posterior mean is well defined almost surely and satisfies
\begin{equation}
  \sup_{f_0\in\mathcal F_n}\mathrm E_{f_0}^{(n)}\mathrm E_{\mathrm{MC}}
  d^2(\widehat f_{n,N}^{\mathrm{ABC}},f_0)
  \le C'\epsilon_n^2.
  \label{eq:meta-mc-risk}
\end{equation}
\end{corollary}

\begin{proof}
See Appendix~\ref{app:meta-proofs}.
\end{proof}

Conditional on a positive number of acceptances, the accepted draws are iid from the ideal ABC posterior, and the ideal posterior second moment controls the empirical mean. On the no-acceptance event, the loss is controlled by a uniform bound on $d(\theta^\dagger,f_0)$ over the target class.

\section{Orthogonal-series regression with \texorpdfstring{$g$-and-$k$}{g-and-k} errors}
\label{sec:gk-regression}

\subsection{Fixed-design regression model and coefficient summaries}

Fix $\alpha>0$ and $R>0$, and let
\begin{equation}
  \Theta_\alpha(R)=\left\{\theta\in\ell^2:\sum_{j=1}^\infty j^{2\alpha}\theta_j^2\le R^2\right\}.
  \label{eq:reg-ellipsoid}
\end{equation}
The loss is the squared $\ell^2$ norm. Set
\begin{equation}
  K_n=\left\lceil n^{1/(2\alpha+1)}\right\rceil,
  \qquad
  \epsilon_n=K_n^{-\alpha}+\sqrt{K_n/n},
  \qquad
  \rho_{n,\alpha}=n^{-\alpha/(2\alpha+1)}.
  \label{eq:reg-rate}
\end{equation}
Then $\epsilon_n\asymp\rho_{n,\alpha}$ and $K_n\asymp n\rho_{n,\alpha}^2$.

For $i,j\in\{1,\ldots,n\}$, define the normalized discrete-cosine design of \citet{AhmedNatarajanRao1974} by
\begin{equation}
  \phi_{1,n}(i)=1,
  \qquad
  \phi_{j,n}(i)=\sqrt2\cos\left\{\frac{\pi(j-1)(i-1/2)}{n}\right\},\quad j\ge2.
  \label{eq:dct-design}
\end{equation}
Then
\begin{equation}
  \frac1n\sum_{i=1}^n\phi_{j,n}(i)\phi_{k,n}(i)=\mathbf 1\{j=k\},
  \qquad 1\le j,k\le n,
  \label{eq:dct-orthogonality}
\end{equation}
and $|\phi_{j,n}(i)|\le\sqrt2$.

Fix $g>0$ and $c\in(0,1/2)$. Define the $k=0$ $g$-and-$k$ quantile transformation and its centered version by
\begin{equation}
  T_{g,c}(z)=z\left\{1+c\tanh(gz/2)\right\},
  \qquad
  \mu_{g,c}=\mathrm E\{T_{g,c}(G)\},
  \qquad
  h_{g,c}(z)=T_{g,c}(z)-\mu_{g,c},
  \label{eq:implicit-error-map}
\end{equation}
where $G\sim\mathrm N(0,1)$, and put
\begin{equation}
  \varepsilon=h_{g,c}(G).
  \label{eq:implicit-error}
\end{equation}
The model treats $\mu_{g,c}$ as a known constant. It may be precomputed to arbitrary numerical accuracy by one-dimensional quadrature; thereafter, each error draw requires only a Gaussian draw and evaluation of $T_{g,c}$. Before centering, $T_{g,c}$ is the $A=0$, $B=1$, $k=0$ subfamily of the usual $g$-and-$k$ quantile family; see \citet{HaynesMacGillivrayMengersen1997} and \citet{RaynerMacGillivray2002}.

The regression model is
\begin{equation}
  X_i=\sum_{j=1}^n\theta_{0j}\phi_{j,n}(i)+\varepsilon_i,
  \qquad i=1,\ldots,n,
  \label{eq:implicit-regression-model}
\end{equation}
where the errors are iid copies of $\varepsilon$ and $\theta_0\in\Theta_\alpha(R)$. The simulator draws independent standard Gaussian variables and applies $h_{g,c}$. By contrast, likelihood evaluation at a residual $x$ requires the inverse of $h_{g,c}$ through
\begin{equation}
  q_{g,c}(x)=\frac{\varphi\{h_{g,c}^{-1}(x)\}}{h_{g,c}'\{h_{g,c}^{-1}(x)\}},
  \label{eq:implicit-error-density}
\end{equation}
where $\varphi$ is the standard Gaussian density and $h_{g,c}^{-1}(x)$ is obtained by numerically solving $h_{g,c}(z)=x$.

\begin{proposition}[Regularity of the centered $g$-and-$k$ error]
\label{prop:implicit-error}
The law in \eqref{eq:implicit-error} is centered and sub-Gaussian. Its density is strictly positive and twice continuously differentiable. There are $t_0,C_q>0$ such that
\begin{equation}
  \mathrm{KL}\{q_{g,c}(\cdot-u),q_{g,c}(\cdot-v)\}\le C_q(u-v)^2
  \label{eq:implicit-shift-kl}
\end{equation}
whenever $|u-v|\le t_0$.
\end{proposition}

\begin{proof}
See Appendix~\ref{app:gk-proofs}.
\end{proof}

Define
\begin{equation}
  S_{n,K_n}(X)_j=\frac1n\sum_{i=1}^n\phi_{j,n}(i)X_i,
  \qquad j=1,\ldots,K_n.
  \label{eq:implicit-reg-summary}
\end{equation}
Under $\mathrm P_\theta^{(n)}$, the summary satisfies $\mathrm E_\theta^{(n)}S_{n,K_n}(X)=\bm\theta_{K_n}$ by \eqref{eq:dct-orthogonality}.

\subsection{Fixed Gaussian series prior}

Let $0<\lambda_-<\lambda_+<\infty$, and let $\nu$ be a probability measure on $(0,\infty)$ satisfying
\begin{equation}
  \nu([\lambda_-,\lambda_+])=1.
  \label{eq:reg-scale-hyperprior}
\end{equation}
The prior $\Pi$ is the marginal law of the coefficient sequence $\theta=(\theta_j)_{j\ge1}$ under the hierarchy
\begin{equation}
  \Lambda\sim\nu,
  \qquad
  \theta_j=\Lambda j^{-\alpha-1/2}G_j,
  \qquad j\ge1,
  \label{eq:reg-fixed-prior}
\end{equation}
where the $G_j$ are independent standard Gaussian variables, independent of $\Lambda$. Since $\sum_{j=1}^{\infty} j^{-2\alpha-1}<\infty$, the series belongs to $\ell^2$ almost surely. We take $\Pi_n=\Pi$ for every $n$. The prior does not depend on $R$; its smoothness exponent is matched to $\alpha$. To generate the regression dataset in \eqref{eq:implicit-regression-model}, only $\Lambda$ and the first $n$ coefficients need be generated; the same coordinates suffice for the acceptance decision.

Synthetic data $Z$ are generated from \eqref{eq:implicit-regression-model} at a draw $\theta\sim\Pi$, and
\begin{equation}
  L_{n,\mathrm{reg}}^{\mathrm{ABC}}(\theta;X)
  =\mathrm P_\theta^{(n)}\left\{\left\lVert S_{n,K_n}(Z)-S_{n,K_n}(X)\right\rVert_2\le A\epsilon_n\right\}.
  \label{eq:implicit-reg-abc-likelihood}
\end{equation}
Let $\Pi_{n,\mathrm{reg}}^{\mathrm{ABC}}(\cdot\mid X)$ be the corresponding ideal ABC posterior. Let $\widehat\Pi_{n,N,\mathrm{reg}}^{\mathrm{ABC}}(\cdot\mid X)$ denote the corresponding empirical rejection-ABC posterior. Both posteriors use the fallback point zero under the conventions of Section~\ref{sec:notation}.

\subsection{Ideal ABC posterior upper bound}

\begin{theorem}[$g$-and-$k$ error regression ABC contraction]
\label{thm:implicit-reg-contract}
There is $A_0<\infty$ such that, for every $A>A_0$, there is $M_0<\infty$ for which, for every fixed $M\ge M_0$,
\begin{equation}
  \sup_{\theta_0\in\Theta_\alpha(R)}\mathrm E_{\theta_0}^{(n)}
  \Pi_{n,\mathrm{reg}}^{\mathrm{ABC}}
  \left\{\theta:\lVert\theta-\theta_0\rVert_{\ell^2}>M\rho_{n,\alpha}\mid X\right\}\to0.
  \label{eq:implicit-reg-contract}
\end{equation}
\end{theorem}

\begin{proof}
See Appendix~\ref{app:gk-proofs}.
\end{proof}

Let $r_n=\sqrt{K_n}$ and let $\mathcal C_n:\ell^2\to\ell^2$ be the metric projection onto the closed ball of radius $r_n$. Define the clipped posterior mean
\begin{equation}
  \widehat\theta_{n,\mathrm{reg}}^{\mathrm{clip}}
  =\int_{\ell^2}\mathcal C_n(\theta)\Pi_{n,\mathrm{reg}}^{\mathrm{ABC}}(\mathrm d\theta\mid X).
  \label{eq:implicit-clipped-mean}
\end{equation}
Radial clipping depends on the full norm $\lVert\theta\rVert_{\ell^2}$. The clipped mean and its empirical counterpart below are idealized $\ell^2$-valued estimators: finite-coordinate simulation suffices for the data and acceptance decision, but does not determine the full norm exactly.

\begin{corollary}[$g$-and-$k$ error regression clipped posterior mean risk]
\label{cor:implicit-reg-risk}
Under the conditions of Theorem~\ref{thm:implicit-reg-contract},
\begin{equation}
  \sup_{\theta_0\in\Theta_\alpha(R)}\mathrm E_{\theta_0}^{(n)}
  \left\lVert\widehat\theta_{n,\mathrm{reg}}^{\mathrm{clip}}-\theta_0\right\rVert_{\ell^2}^2
  \le Cn^{-2\alpha/(2\alpha+1)}.
  \label{eq:implicit-reg-risk}
\end{equation}
\end{corollary}

\begin{proof}
See Appendix~\ref{app:gk-proofs}.
\end{proof}

\subsection{Minimax lower bound}

\begin{theorem}[$g$-and-$k$ error regression minimax lower bound]
\label{thm:implicit-reg-lower}
There exists $c>0$ such that, for all sufficiently large $n$,
\begin{equation}
  \inf_{\widetilde\theta_n}\sup_{\theta_0\in\Theta_\alpha(R)}
  \mathrm E_{\theta_0}^{(n)}\lVert\widetilde\theta_n-\theta_0\rVert_{\ell^2}^2
  \ge c n^{-2\alpha/(2\alpha+1)},
  \label{eq:implicit-reg-lower}
\end{equation}
where the infimum is over all estimators based on \eqref{eq:implicit-regression-model}.
\end{theorem}

\begin{proof}
See Appendix~\ref{app:gk-proofs}.
\end{proof}

Thus the clipped ABC posterior mean is minimax-rate optimal in this setting.

\subsection{Monte Carlo estimators}

\begin{corollary}[$g$-and-$k$ error regression Monte Carlo contraction]
\label{cor:implicit-reg-mc}
Under the conditions of Theorem~\ref{thm:implicit-reg-contract}, there is $\kappa<\infty$ such that, if $N_n e^{-\kappa K_n}\to\infty$, then, for every fixed $M\ge M_0$, where $M_0$ is as in Theorem~\ref{thm:implicit-reg-contract},
\begin{equation}
  \sup_{\theta_0\in\Theta_\alpha(R)}\mathrm E_{\theta_0}^{(n)}\mathrm E_{\mathrm{MC}}
  \widehat\Pi_{n,N,\mathrm{reg}}^{\mathrm{ABC}}
  \left\{\theta:\lVert\theta-\theta_0\rVert_{\ell^2}>M\rho_{n,\alpha}\mid X\right\}\to0.
  \label{eq:implicit-reg-mc-contract}
\end{equation}
\end{corollary}

\begin{proof}
See Appendix~\ref{app:gk-proofs}.
\end{proof}

\begin{corollary}[$g$-and-$k$ error regression empirical clipped posterior mean risk]
\label{cor:implicit-reg-mc-risk}
Under the conditions of Corollary~\ref{cor:implicit-reg-mc}, define
\begin{equation}
  \widehat\theta_{n,N,\mathrm{reg}}^{\mathrm{clip}}
  =\int_{\ell^2}\mathcal C_n(\theta)\widehat\Pi_{n,N,\mathrm{reg}}^{\mathrm{ABC}}(\mathrm d\theta\mid X).
  \label{eq:implicit-reg-mc-barycenter}
\end{equation}
If $N_n e^{-\kappa K_n}\ge C_N\log(1/\rho_{n,\alpha})$, where $C_N$ is sufficiently large, then
\begin{equation}
  \sup_{\theta_0\in\Theta_\alpha(R)}\mathrm E_{\theta_0}^{(n)}\mathrm E_{\mathrm{MC}}
  \left\lVert\widehat\theta_{n,N,\mathrm{reg}}^{\mathrm{clip}}-\theta_0\right\rVert_{\ell^2}^2
  \le Cn^{-2\alpha/(2\alpha+1)}.
  \label{eq:implicit-reg-mc-risk}
\end{equation}
\end{corollary}

\begin{proof}
See Appendix~\ref{app:gk-proofs}.
\end{proof}

\section{Nonparametric decompounding from aggregate observations}
\label{sec:density}

\subsection{Compound Poisson observations}

Fix $d\in\mathbb N$ and write $\mathcal X=[0,1]^d$. Let
\begin{equation}
  \mathcal D=\left\{p\in\mathcal L^2(\mathcal X):p\ge0\ \text{a.e.},\ \int_{\mathcal X}p(x)\,\mathrm dx=1\right\}
  \label{eq:density-space}
\end{equation}
be the ambient space of jump densities. Let
\begin{equation}
  \mathcal I=[\underline\beta,\overline\beta],
  \qquad 0<\underline\beta<\overline\beta<\pi/2,
  \qquad \Theta=\mathcal I\times\mathcal D.
  \label{eq:decompound-parameter-space}
\end{equation}
Equip $\Theta$ with the product Hilbert metric
\begin{equation}
  d_\Theta\left((\beta,p),(\widetilde\beta,\widetilde p)\right)
  =\left(|\beta-\widetilde\beta|^2+\lVert p-\widetilde p\rVert_2^2\right)^{1/2}.
  \label{eq:decompound-product-metric}
\end{equation}
The count intensity $\beta_0\in \mathcal I$ is unknown, although the estimation target is $p_0$, with squared $\mathcal L^2$ loss; $\beta_0$ is a nuisance parameter. For each observation period, let $C_i\sim\operatorname{Poisson}(\beta_0)$ and let $Y_{i1},Y_{i2},\ldots$ be iid with density $p_0\in\mathcal D$. Assume independence between all counts and jumps and across periods. We observe only the totals
\begin{equation}
  X_i=\sum_{k=1}^{C_i}Y_{ik},\qquad i=1,\ldots,n,
  \label{eq:decompound-observations}
\end{equation}
with an empty sum equal to zero. This is the vector-valued counterpart of the compound Poisson construction in \citet[Section~1.2]{CekanaviciusNovak2025}; see also \citet{GugushviliVanderMeulenSpreij2015}. In one dimension, this models aggregate amounts from a Poisson number of individual contributions with a known bounded range.

Write $\mathrm Q_{\beta,p}$ for the law of one total at intensity $\beta$ and jump density $p$. Extending $p$ by zero outside $\mathcal X$, conditioning on the count gives
\begin{equation}
\begin{aligned}
  \mathrm Q_{\beta,p}(\mathrm dx)&=e^{-\beta}\delta_0(\mathrm dx)+q_{\beta,p}(x)\,\mathrm dx,\\
  q_{\beta,p}(x)&=e^{-\beta}\sum_{r=1}^{\infty}\frac{\beta^r}{r!}\,p^{(*r)}(x),
\end{aligned}
\label{eq:decompound-law}
\end{equation}
where $p^{(*r)}$ is the $r$-fold convolution on $\mathbb R^d$. The continuous component has mass $1-e^{-\beta}$. Since the jumps have a density and are nonnegative, $X_i=0$ if and only if $C_i=0$, almost surely. The zero-count indicator $\mathbf 1\{C_i=0\}$ is observable as $\mathbf 1\{X_i=0\}$. Thus $\mathrm Q_{\beta,p}(\{0\})=e^{-\beta}$, independently of $p$. At a candidate $(\beta,p)$, each nonzero observation contributes $q_{\beta,p}(X_i)$ to the likelihood, whereas a zero observation contributes $e^{-\beta}$. Simulation at $(\beta,p)$ instead draws a $\operatorname{Poisson}(\beta)$ count, generates that many independent jumps from $p$, and sums them.

Nonparametric decompounding estimators based on Fourier inversion are studied by \citet{VanEsGugushviliSpreij2007}. \citet{GugushviliVanderMeulenSpreij2015} establish Bayesian contraction in a Hellinger metric on compound Poisson observation laws; \citet{GugushviliVanderMeulenSpreij2018} also develop a likelihood-based implementation using latent-variable augmentation and Markov chain Monte Carlo. \citet{GoffardLaub2021} use ABC with sequential Monte Carlo and Wasserstein discrepancies to fit and compare insurance loss models from aggregate data. Here the target is the jump density in $\mathcal L^2$ loss, and the ABC comparison uses trigonometric summaries together with the proportion of zero-count periods. The method avoids evaluation of \eqref{eq:decompound-law}.

For every sample size $r$, let $\mathrm P_{\beta,p}^{(r)}=\mathrm Q_{\beta,p}^{\otimes r}$ be the joint law of $r$ independent totals with common law $\mathrm Q_{\beta,p}$, and let $\mathrm E_{\beta,p}^{(r)}$ denote expectation under this law. Convolution with a probability density does not increase the $\mathcal L^1$ norm, by the triangle inequality and Tonelli's theorem. For $r\ge1$, the identity
\begin{equation*}
  p^{(*(r+1))}-\widetilde p^{(*(r+1))}
  =\left(p^{(*r)}-\widetilde p^{(*r)}\right)*p
   +\widetilde p^{(*r)}*(p-\widetilde p)
\end{equation*}
therefore gives $\lVert p^{(*r)}-\widetilde p^{(*r)}\rVert_1\le r\lVert p-\widetilde p\rVert_1$ by induction, starting from $r=1$. Summing this bound with the Poisson weights in \eqref{eq:decompound-law} yields
\begin{equation}
  \mathrm{TV}(\mathrm Q_{\beta,p},\mathrm Q_{\beta,\widetilde p})
  \le\frac{\beta}{2}\lVert p-\widetilde p\rVert_1
  \le\frac{\beta}{2}\lVert p-\widetilde p\rVert_2.
\end{equation}
For the same jump density, couple counts with intensities $\beta\le\widetilde\beta$ by adding an independent $\operatorname{Poisson}(\widetilde\beta-\beta)$ count and use a common jump sequence. The totals agree whenever the added count is zero. Reversing the roles of the intensities when necessary therefore gives
\begin{equation}
  \mathrm{TV}(\mathrm Q_{\beta,p},\mathrm Q_{\widetilde\beta,p})
  \le1-e^{-|\beta-\widetilde\beta|}
  \le|\beta-\widetilde\beta|.
\end{equation}
Combining the two comparisons yields
\begin{equation}
  \mathrm{TV}(\mathrm Q_{\beta,p},\mathrm Q_{\widetilde\beta,\widetilde p})
  \le |\beta-\widetilde\beta|
      +\frac{\overline\beta}{2}\lVert p-\widetilde p\rVert_2.
\end{equation}
Consequently, this model defines measurable probability kernels on $\Theta$. Unless stated otherwise, all bounds below are uniform over $\beta_0\in \mathcal I$ and the specified jump-density class; constants may depend on the fixed endpoints of $\mathcal I$.

\subsection{Periodic Sobolev class and trigonometric spaces}

Let
\begin{equation}
  t_0(x)=1,
  \qquad
  t_{2q-1}(x)=\sqrt2\cos(2\pi qx),
  \qquad
  t_{2q}(x)=\sqrt2\sin(2\pi qx),\quad q\ge1,
\end{equation}
and, for $\boldsymbol j=(j_1,\ldots,j_d)\in\mathbb N_0^d$, set
\begin{equation}
  \varphi_{\boldsymbol j}(x)=\prod_{r=1}^d t_{j_r}(x_r),
  \qquad
  q(\boldsymbol j)=\max_{1\le r\le d}\left\lceil j_r/2\right\rceil.
  \label{eq:tensor-trigonometric-basis}
\end{equation}
These basis functions are products of one-dimensional sine and cosine functions, with a constant factor allowed in each coordinate. Enumerate the constant first and the remaining functions by nondecreasing $q(\boldsymbol j)$, writing $q_j$ for the frequency-block index of $\varphi_j$. Let
\begin{equation}
  \mathcal V_J=\operatorname{span}\{\varphi_j:q_j\le J\},
  \qquad
  D_J=\dim(\mathcal V_J)=(2J+1)^d.
  \label{eq:density-trig-space}
\end{equation}
The product of two elements of $\mathcal V_J$ belongs to $\mathcal V_{2J}$.

For $p\in\mathcal D$, write $\theta_j(p)=\int_{\mathcal X}\varphi_j(x)p(x)\,\mathrm dx$. Fix $\alpha>d/2$, $R>0$, and $0<m<1<M<\infty$, and define
\begin{equation}
  \mathcal P_\alpha(R,m,M)=\left\{p\in\mathcal D:\begin{aligned}
  &m\le p\le M\ \text{a.e.},\\
  &\sum_{j=2}^{\infty}(1+q_j)^{2\alpha}\theta_j(p)^2\le R^2
  \end{aligned}\right\}.
  \label{eq:density-class}
\end{equation}
Since $\theta_1(p)=1$, adding one to the weighted coefficient sum in \eqref{eq:density-class} gives a quantity equivalent to $\lVert p\rVert_{\mathcal H^\alpha}^2$ under the block ordering. Here $\mathcal H^\alpha([0,1]^d)$ denotes the periodic $\mathcal L^2$-Sobolev space. For $\alpha>d/2$, it is a Banach algebra, and $\{\sqrt p:p\in\mathcal P_\alpha(R,m,M)\}$ is bounded in $\mathcal H^\alpha$; see \citet{RunstSickel1996} and Appendix~\ref{app:external}.

Set
\begin{equation}
  J_n=\left\lceil n^{1/(2\alpha+d)}\right\rceil,
  \qquad
  \widetilde K_n=D_{J_n}-1,
  \qquad
  K_n=D_{2J_n}-1,
  \label{eq:density-dimensions}
\end{equation}
\begin{equation}
  \epsilon_n=J_n^{-\alpha}+\sqrt{K_n/n},
  \qquad
  \rho_{n,\alpha,d}=n^{-\alpha/(2\alpha+d)},
  \qquad
  m_n=\left\lceil nD_{J_n}\right\rceil.
  \label{eq:density-rate}
\end{equation}
Then $\widetilde K_n\asymp K_n\asymp J_n^d$, $\epsilon_n\asymp\rho_{n,\alpha,d}$, and $K_n\asymp n\rho_{n,\alpha,d}^2$.

\subsection{Zero-count and trigonometric summaries of aggregate observations}

Extend each trigonometric basis function periodically to $\mathbb R^d$. For an integer $J\ge1$, define
\begin{equation}
\begin{aligned}
  \psi_{2J}(x)&=(\varphi_2(x),\ldots,\varphi_{D_{2J}}(x))^\top,\\
  \eta_J(x)&=\mathbf 1\{x\ne0\}\psi_{2J}(x),
  \qquad
  s_J(\beta,p)=\int_{\mathbb R^d}\eta_J(x)\mathrm Q_{\beta,p}(\mathrm dx).
\end{aligned}
\label{eq:decompound-summary-center}
\end{equation}
The indicator $\mathbf 1\{x\ne0\}$ removes the contribution of the atom from the trigonometric coordinates. Retain the information in that atom in a separate coordinate by defining
\begin{equation}
  \xi_J(x)=\begin{pmatrix}\mathbf 1\{x=0\}\\ \eta_J(x)\end{pmatrix},
  \qquad
  t_J(\beta,p)=\int_{\mathbb R^d}\xi_J(x)\mathrm Q_{\beta,p}(\mathrm dx)
  =\begin{pmatrix}e^{-\beta}\\ s_J(\beta,p)\end{pmatrix}.
  \label{eq:decompound-joint-summary-center}
\end{equation}
The observed and synthetic summaries are
\begin{equation}
  S_n(X)=\frac1n\sum_{i=1}^n\xi_{J_n}(X_i),
  \qquad
  S_{m_n}(Z)=\frac1{m_n}\sum_{i=1}^{m_n}\xi_{J_n}(Z_i),
  \label{eq:density-summary}
\end{equation}
where the $Z_i$ are iid synthetic totals with law $\mathrm Q_{\beta,p}$. In the notation of Section~\ref{sec:notation}, $\mathrm Q_{\beta,p}^{(n)}=\mathrm P_{\beta,p}^{(m_n)}$. The ABC likelihood is
\begin{equation}
  L_n^{\mathrm{ABC}}(\beta,p;X)
  =\mathrm P_{\beta,p}^{(m_n)}\left\{\left\lVert S_{m_n}(Z)-S_n(X)\right\rVert_2\le A\epsilon_n\right\}.
  \label{eq:density-abc-likelihood}
\end{equation}
The summaries have $K_n+1$ coordinates, which is of the same order as $K_n$. Their first coordinates are proportions of zero-count periods, and both averages use all observation periods, including zero-count periods. The larger synthetic sample controls variation uniformly over the square-root prior specified below, whose density envelope grows with $J_n$. The fluctuation scale of $S_n(X)$ around $t_{J_n}(\beta_0,p_0)$ remains $\sqrt{K_n/n}$.

\subsection{Square-root Gaussian series prior and intensity prior}

Use the scale hyperprior in \eqref{eq:reg-scale-hyperprior}. Conditional on $\Lambda$, draw independent coefficients
\begin{equation}
  v_j=\Lambda(1+q_j)^{-\alpha-d/2}G_j,
  \qquad j=2,\ldots,D_{J_n},
  \label{eq:density-square-root-coefficients}
\end{equation}
set
\begin{equation}
  g_v=1+\sum_{j=2}^{D_{J_n}}v_j\varphi_j,
  \qquad
  p_v=\frac{g_v^2}{\lVert g_v\rVert_{\mathcal L^2}^2}.
  \label{eq:density-square-root-prior}
\end{equation}
By orthonormality, $\lVert g_v\rVert_2^2=1+\sum_{j=2}^{D_{J_n}}v_j^2$, so the normalizing constant is explicit. Let $\Gamma_n$ be the induced coefficient law of $v$ and $\Pi_n=\Gamma_n\circ(v\mapsto p_v)^{-1}$ its pushforward to $\mathcal D$. Every density $p_v$ is nonnegative and integrates to one, even when $g_v$ changes sign. We take $\operatorname{supp}(\Pi_n)$ with respect to the $\mathcal L^2$ topology. Since $\mathcal V_{2J_n}$ is finite-dimensional and therefore closed in $\mathcal L^2$, and since the constraints $0\le p\le D_{J_n}$ almost everywhere are preserved under $\mathcal L^2$ limits, every $p\in\operatorname{supp}(\Pi_n)$ belongs to $\mathcal V_{2J_n}$ and satisfies $0\le p\le D_{J_n}$ almost everywhere. The fallback density is $p^\dagger=1$.

Independently of the density coefficients and their scale, assign $\beta$ a simulable probability density $w$ on $\mathcal I$ satisfying
\begin{equation}
  0<w_-\le w(\beta)\le w_+<\infty\quad\text{for almost every }\beta\in \mathcal I.
  \label{eq:decompound-intensity-prior}
\end{equation}
The uniform density on $\mathcal I$ is one such choice. The joint prior is
\begin{equation}
  \widetilde\Pi_n(\mathrm d\beta,\mathrm dp)
  =w(\beta)\,\mathrm d\beta\,\Pi_n(\mathrm dp).
  \label{eq:decompound-joint-prior}
\end{equation}
Let $\widetilde\Pi_n^{\mathrm{ABC}}$ and $\widehat{\widetilde\Pi}_{n,N}^{\mathrm{ABC}}$ be the joint ideal and empirical posteriors defined by Section~\ref{sec:notation}, using this prior, the likelihood in \eqref{eq:density-abc-likelihood}, and a fixed fallback point $(\beta^\dagger,1)$ with $\beta^\dagger\in \mathcal I$. The joint ABC acceptance normalizer is
\begin{equation}
  D_n(X)=\int_{\mathcal I}\int_{\mathcal D}
  L_n^{\mathrm{ABC}}(\beta,p;X)\Pi_n(\mathrm dp)w(\beta)\,\mathrm d\beta.
  \label{eq:decompound-joint-normalizer}
\end{equation}
For a Borel set $B\subset\mathcal D$, write
\begin{equation}
\begin{aligned}
  \Pi_n^{\mathrm{ABC}}(B\mid X)
  &=\widetilde\Pi_n^{\mathrm{ABC}}(\mathcal I\times B\mid X),\\
  \widehat\Pi_{n,N}^{\mathrm{ABC}}(B\mid X)
  &=\widehat{\widetilde\Pi}_{n,N}^{\mathrm{ABC}}(\mathcal I\times B\mid X).
\end{aligned}
\label{eq:decompound-density-marginals}
\end{equation}
The contraction and risk statements below concern these density marginals. The intensity is integrated out, rather than held fixed or replaced by a point estimate.

\begin{proposition}[Sampling from a square-root series density]
\label{prop:density-simulator}
For every $v$,
\begin{equation}
  0\le p_v(x)\le D_{J_n},\qquad x\in[0,1]^d.
  \label{eq:density-envelope}
\end{equation}
Consequently, exact iid sampling from $p_v$ is obtained as follows: draw $U\sim\mathrm{Unif}([0,1]^d)$ and $V\sim\mathrm{Unif}(0,1)$ independently, and accept $U$ when $V\le p_v(U)/D_{J_n}$. The acceptance probability is $D_{J_n}^{-1}$.
\end{proposition}

\begin{proof}
See Appendix~\ref{app:density-proofs}.
\end{proof}

For $p_0\in\mathcal P_\alpha(R,m,M)$, put
\begin{equation}
  c_0=\int_{\mathcal X}\sqrt{p_0(x)}\,\mathrm dx,
  \qquad
  g_0=\frac{\sqrt{p_0}}{c_0}.
  \label{eq:density-target-square-root}
\end{equation}
The constant coefficient of $g_0$ is one and $p_0=g_0^2/\lVert g_0\rVert_2^2$. Let $g_{0,J}$ be the orthogonal projection of $g_0$ onto $\mathcal V_J$ and let $v_{0,J}$ be its nonconstant coefficient vector.

\begin{proposition}[Square-root approximation and local prior mass]
\label{prop:density-local-mass}
There are constants $C,\delta>0$ such that, for all sufficiently large $n$, uniformly over $p_0\in\mathcal P_\alpha(R,m,M)$,
\begin{equation}
  \left\lVert\frac{g_{0,J_n}^2}{\lVert g_{0,J_n}\rVert_2^2}-p_0\right\rVert_2
  \le CJ_n^{-\alpha},
  \label{eq:density-square-root-bias}
\end{equation}
and
\begin{equation}
  \Gamma_n\left\{v:\lVert v-v_{0,J_n}\rVert_2\le\delta\epsilon_n\right\}
  \ge e^{-CK_n}.
  \label{eq:density-square-root-local-mass}
\end{equation}
Every coefficient vector in this ball induces a density within $C\epsilon_n$ of $p_0$ in $\mathcal L^2$. Consequently, the image of this ball is a local set satisfying the same lower bound under $\Pi_n$.
\end{proposition}

\begin{proof}
See Appendix~\ref{app:density-proofs}.
\end{proof}

\subsection{Stability and concentration of aggregate summaries}

Write $\mathbf i$ for the imaginary unit, so that $\mathbf i^2=-1$. For $\boldsymbol k\in\mathbb Z^d$, define the jump Fourier coefficient
\begin{equation}
  a_{\boldsymbol k}(p)=\int_{\mathcal X}e^{2\pi\mathbf i\boldsymbol k^\top u}p(u)\,\mathrm du.
\end{equation}
The characteristic-function identity for compound Poisson sums, as used in \citet{VanEsGugushviliSpreij2007}, gives
\begin{equation}
\begin{aligned}
  b_{\boldsymbol k}(\beta,p)
  &=\int_{\mathbb R^d}\mathbf 1\{x\ne0\}e^{2\pi\mathbf i\boldsymbol k^\top x}\mathrm Q_{\beta,p}(\mathrm dx)\\
  &=F_\beta\{a_{\boldsymbol k}(p)\},
  \qquad F_\beta(z)=e^{-\beta}(e^{\beta z}-1).
\end{aligned}
\label{eq:decompound-fourier-transform}
\end{equation}
Indeed, conditionally on $C_i=r$, the characteristic function of the total is $a_{\boldsymbol k}(p)^r$; summing over $r\ge1$ proves the identity. Periodic folding maps $x\in\mathbb R^d$ to its coordinatewise fractional parts in $[0,1)^d$. Under this map, the absolutely continuous part $q_{\beta,p}(x)\,\mathrm dx$ of $\mathrm Q_{\beta,p}$ has the subprobability density $r_{\beta,p}^\circ(u)=\sum_{a\in\mathbb Z^d}q_{\beta,p}(u+a)$, of total mass $1-e^{-\beta}$; see Lemma~\ref{lem:decompound-folding}. Since the trigonometric functions are periodic, the real coordinates of $s_J(\beta,p)$ are precisely the coefficients of $r_{\beta,p}^\circ$ in the nonconstant product trigonometric basis.

For $J\ge0$, let $\mathrm{Proj}_J$ be the orthogonal projection onto $\mathcal V_J$.

\begin{proposition}[Stability of aggregate Fourier summaries]
\label{prop:decompound-stability}
Let $0<\beta<\pi/2$ and put $c_\beta=\beta e^{-2\beta}\cos\beta>0$. For every $p,\widetilde p\in\mathcal D$ and integer $J\ge1$,
\begin{equation}
  c_\beta\lVert \mathrm{Proj}_{2J}(p-\widetilde p)\rVert_2
  \le\lVert s_J(\beta,p)-s_J(\beta,\widetilde p)\rVert_2
  \le\beta\lVert \mathrm{Proj}_{2J}(p-\widetilde p)\rVert_2.
  \label{eq:decompound-summary-stability}
\end{equation}
\end{proposition}

\begin{proof}
See Appendix~\ref{app:density-proofs}.
\end{proof}

\begin{proposition}[Stability with unknown count intensity]
\label{prop:decompound-joint-stability}
There is a constant $C_{\mathrm{st}}\ge1$, depending only on $\mathcal I$ and $M$, such that, for every $\beta,\beta_0\in \mathcal I$, $p,p_0\in\mathcal D$ with $p_0\le M$ almost everywhere, and integer $J\ge1$,
\begin{equation}
\begin{aligned}
 C_{\mathrm{st}}^{-1}
 \left(|\beta-\beta_0|+\lVert\mathrm{Proj}_{2J}(p-p_0)\rVert_2\right)
 &\le \lVert t_J(\beta,p)-t_J(\beta_0,p_0)\rVert_2\\
 &\le C_{\mathrm{st}}
 \left(|\beta-\beta_0|+\lVert\mathrm{Proj}_{2J}(p-p_0)\rVert_2\right).
\end{aligned}
\label{eq:decompound-joint-stability}
\end{equation}
\end{proposition}

\begin{proof}
See Appendix~\ref{app:density-proofs}.
\end{proof}

The constant in \eqref{eq:decompound-joint-stability} is independent of the summary dimension. Only $p_0$, not the candidate density $p$, is required to satisfy the fixed upper bound $M$. The intensity restriction is sufficient for the inverse bounds in \eqref{eq:decompound-summary-stability} and \eqref{eq:decompound-joint-stability}; it is not an identifiability threshold for the full decompounding model. The collection $(t_J(\beta,p))_{J\ge1}$ identifies $(\beta,p)$: the zero-count probability determines $\beta$, and Proposition~\ref{prop:decompound-stability} then determines every Fourier coefficient of $p$. A fixed finite summary need not be sufficient, and the results below concern contraction around the latent density rather than approximation of the full likelihood posterior.

\begin{proposition}[Observed and synthetic aggregate-summary concentration]
\label{prop:density-concentration}
There are constants $B_0,c_{\mathrm{coef}}>0$ such that, for every fixed $B\ge B_0$ and all sufficiently large $n$,
\begin{align}
  \sup_{\substack{\beta_0\in \mathcal I\\p_0\in\mathcal P_\alpha(R,m,M)}}
  \mathrm P_{\beta_0,p_0}^{(n)}\left\{\lVert S_n(X)-t_{J_n}(\beta_0,p_0)\rVert_2>B\sqrt{K_n/n}\right\}
  &\le e^{-c_{\mathrm{coef}}B^2K_n},
  \label{eq:density-observed-concentration}\\
  \sup_{\substack{\beta\in \mathcal I\\p\in\operatorname{supp}(\Pi_n)}}
  \mathrm P_{\beta,p}^{(m_n)}\left\{\lVert S_{m_n}(Z)-t_{J_n}(\beta,p)\rVert_2>B\sqrt{K_n/n}\right\}
  &\le e^{-c_{\mathrm{coef}}B^2K_n}.
  \label{eq:density-synthetic-concentration}
\end{align}
\end{proposition}

\begin{proof}
See Appendix~\ref{app:density-proofs}.
\end{proof}

\begin{corollary}[Uniform synthetic aggregate-summary small ball]
\label{cor:density-synthetic-small-ball}
For the constant $B_0$ in Proposition~\ref{prop:density-concentration}, there is $c_1>0$ such that, for all sufficiently large $n$,
\begin{equation}
  \inf_{\substack{\beta\in \mathcal I\\p\in\operatorname{supp}(\Pi_n)}}
  \mathrm P_{\beta,p}^{(m_n)}\left\{\lVert S_{m_n}(Z)-t_{J_n}(\beta,p)\rVert_2\le B_0\sqrt{K_n/n}\right\}
  \ge c_1.
  \label{eq:density-synthetic-small-ball}
\end{equation}
\end{corollary}

\begin{proof}
See Appendix~\ref{app:density-proofs}.
\end{proof}

\subsection{Ideal ABC posterior upper bound}

\begin{theorem}[Decompounding ideal ABC contraction]
\label{thm:density-ideal}
For the compound Poisson model, prior, summaries, and sample sizes specified above, with $\mathcal I$ as in \eqref{eq:decompound-parameter-space} and $\alpha>d/2$, there is $A_0<\infty$ such that, for every $A>A_0$, there exists $L_0<\infty$ for which, for every fixed $L\ge L_0$,
\begin{equation}
  \sup_{\substack{\beta_0\in \mathcal I\\p_0\in\mathcal P_\alpha(R,m,M)}}\mathrm E_{\beta_0,p_0}^{(n)}
  \Pi_n^{\mathrm{ABC}}
  \left\{p:\lVert p-p_0\rVert_2>L\rho_{n,\alpha,d}\mid X\right\}\to0.
  \label{eq:density-contract}
\end{equation}
\end{theorem}

\begin{proof}
See Appendix~\ref{app:density-proofs}.
\end{proof}

\begin{proposition}[Jump-density posterior second moment bound]
\label{prop:density-second-moment}
Under the conditions of Theorem~\ref{thm:density-ideal}, there exists $C<\infty$ such that
\begin{equation}
  \sup_{\substack{\beta_0\in \mathcal I\\p_0\in\mathcal P_\alpha(R,m,M)}}\mathrm E_{\beta_0,p_0}^{(n)}
  \int_{\mathcal D}\lVert p-p_0\rVert_2^2\Pi_n^{\mathrm{ABC}}(\mathrm dp\mid X)
  \le Cn^{-2\alpha/(2\alpha+d)}.
  \label{eq:density-second}
\end{equation}
\end{proposition}

\begin{proof}
See Appendix~\ref{app:density-proofs}.
\end{proof}

Define the ideal and empirical posterior mean jump densities by
\begin{equation}
  \widehat p_n^{\mathrm{ABC}}=\int_{\mathcal D}p\,\Pi_n^{\mathrm{ABC}}(\mathrm dp\mid X),
  \qquad
  \widehat p_{n,N}^{\mathrm{ABC}}=\int_{\mathcal D}p\,\widehat\Pi_{n,N}^{\mathrm{ABC}}(\mathrm dp\mid X).
  \label{eq:density-posterior-means}
\end{equation}

\begin{corollary}[Jump-density posterior mean risk]
\label{cor:density-mean-risk}
Under the conditions of Proposition~\ref{prop:density-second-moment}, the ideal ABC posterior mean density is well defined and satisfies
\begin{equation}
  \sup_{\substack{\beta_0\in \mathcal I\\p_0\in\mathcal P_\alpha(R,m,M)}}\mathrm E_{\beta_0,p_0}^{(n)}
  \lVert\widehat p_n^{\mathrm{ABC}}-p_0\rVert_2^2
  \le Cn^{-2\alpha/(2\alpha+d)}.
  \label{eq:density-mean-risk}
\end{equation}
\end{corollary}

\begin{proof}
See Appendix~\ref{app:density-proofs}.
\end{proof}

\subsection{Minimax lower bound}

To compare the upper bound with the best risk available from all aggregate observations, we use Assouad's hypercube method \citep[Theorem~2.12]{Tsybakov2009}. The proof restricts the unknown-intensity model to one fixed intensity and applies this method to smooth trigonometric perturbations of the uniform jump density, together with the fixed-intensity Kullback--Leibler comparison of \citet[Lemma~1, inequality~(11)]{GugushviliVanderMeulenSpreij2015}.

\begin{theorem}[Decompounding minimax lower bound]
\label{thm:density-lower}
There exists $c>0$ such that, for all sufficiently large $n$,
\begin{equation}
  \inf_{\widetilde p_n}\sup_{\substack{\beta_0\in \mathcal I\\p_0\in\mathcal P_\alpha(R,m,M)}}
  \mathrm E_{\beta_0,p_0}^{(n)}\lVert\widetilde p_n-p_0\rVert_2^2
  \ge c n^{-2\alpha/(2\alpha+d)},
  \label{eq:density-lower}
\end{equation}
where the infimum is over all measurable estimators based on the full aggregate sample in \eqref{eq:decompound-observations}, without access to $\beta_0$.
\end{theorem}

\begin{proof}
See Appendix~\ref{app:density-proofs}.
\end{proof}

Corollary~\ref{cor:density-mean-risk} and Theorem~\ref{thm:density-lower} show that the marginal ABC posterior mean attains the minimax rate for estimating the individual jump density from aggregate observations with unknown intensity. The lower bound is for the full aggregate sample, although the ABC procedure uses its zero-count and periodic summaries. The perturbations used in Appendix~\ref{app:density-proofs} converge uniformly to zero while preserving the integral of the density. Since $m<1<M$, all vertices therefore satisfy $m\le p\le M$ for sufficiently large $n$, so these pointwise bounds are compatible with the same lower-bound rate.

\subsection{Monte Carlo estimators}

\begin{theorem}[Decompounding rejection-ABC Monte Carlo contraction]
\label{thm:density-mc}
Under the conditions of Theorem~\ref{thm:density-ideal}, there is $\kappa<\infty$ such that, if $N_n e^{-\kappa K_n}\to\infty$, then, for every fixed $L\ge L_0$, where $L_0$ is as in Theorem~\ref{thm:density-ideal},
\begin{equation}
  \sup_{\substack{\beta_0\in \mathcal I\\p_0\in\mathcal P_\alpha(R,m,M)}}\mathrm E_{\beta_0,p_0}^{(n)}\mathrm E_{\mathrm{MC}}
  \widehat\Pi_{n,N}^{\mathrm{ABC}}
  \left\{p:\lVert p-p_0\rVert_2>L\rho_{n,\alpha,d}\mid X\right\}\to0.
\end{equation}
\end{theorem}

\begin{proof}
See Appendix~\ref{app:density-proofs}.
\end{proof}

\begin{corollary}[Jump-density empirical posterior mean risk]
\label{cor:density-mc-risk}
Under the conditions of Theorem~\ref{thm:density-mc}, if $N_n e^{-\kappa K_n}\ge C_N\log(1/\rho_{n,\alpha,d})$ with $C_N$ sufficiently large, then
\begin{equation}
  \sup_{\substack{\beta_0\in \mathcal I\\p_0\in\mathcal P_\alpha(R,m,M)}}\mathrm E_{\beta_0,p_0}^{(n)}\mathrm E_{\mathrm{MC}}
  \lVert\widehat p_{n,N}^{\mathrm{ABC}}-p_0\rVert_2^2
  \le Cn^{-2\alpha/(2\alpha+d)}.
  \label{eq:density-mc-risk}
\end{equation}
\end{corollary}

\begin{proof}
See Appendix~\ref{app:density-proofs}.
\end{proof}

Conditional on a proposed $(\beta,p_v)$, the sampler in Proposition~\ref{prop:density-simulator} uses an expected $\beta D_{J_n}$ cube-uniform draws per total and $\beta m_nD_{J_n}$ per synthetic dataset. Since $m_n\asymp nD_{J_n}$ and $D_{J_n}\asymp K_n$, this expected number of inner uniform draws is of order $nK_n^2$, uniformly over $\beta\in \mathcal I$. Averaging over the intensity prior preserves this order.

\section{Discussion}
\label{sec:discussion}

The results give minimax-rate ABC contraction, risk bounds for a clipped posterior mean in regression, and posterior mean risk bounds for the jump density in decompounding. The abstract theorem permits these guarantees via conditions on local prior mass, bounds on acceptance probabilities away from the truth, and control of prior mass outside a sieve.

The procedures assume known $\alpha$ and are not adaptive to unknown smoothness. Adaptive minimax guarantees for these ABC constructions would require choosing the prior regularity and summary resolution jointly. Random-series adaptation under ordinary likelihoods \citep{GhosalLemberVanderVaart2003,ShenGhosal2015} suggests hierarchical constructions, but an ABC analysis would also need to control the tolerance and acceptance probabilities across resolutions. For joint inference on the regression coefficients and unknown error-shape parameters, the summaries would additionally need to be informative about the error distribution.

For decompounding, the proportion of zero-count periods controls the unknown count intensity through $\mathrm Q_{\beta,p}(\{0\})=e^{-\beta}$. The joint stability bound permits minimax-rate estimation of the jump density uniformly over a fixed compact intensity interval in $(0,\pi/2)$. The rate statements concern the density marginal and do not assert an optimal estimation rate for the intensity coefficient. Extending the analysis beyond this intensity range would require a different inverse argument, since the positive-real-part bound used here no longer applies.

Reducing $m_n$ would require sharper control of synthetic-summary variability than the bound obtained from $p\le D_{J_n}$ uniformly over $p\in\operatorname{supp}(\Pi_n)$. Reducing the number of prior draws is a separate problem. MCMC-ABC \citep{MarjoramMolitorPlagnolTavare2003} and sequential Monte Carlo ABC \citep{BeaumontCornuetMarinRobert2009,DelMoralDoucetJasra2012} concentrate simulations in regions supported by the data. Extending the risk guarantees to these algorithms would require bounds for dependent or weighted Monte Carlo averages as the summary dimension increases and the tolerance decreases.

For regression adjustment \citep{BeaumontZhangBalding2002,LiFearnhead2018a,LiFearnhead2018b} and synthetic likelihood \citep{Wood2010,FrazierNottDrovandiKohn2022}, nonparametric guarantees would require control of the adjustment error or Gaussian summary approximation, respectively, at a growing summary dimension. The balance between these errors and the statistical rate could guide the choice of summary resolution and simulation budget.

Frequentist coverage of credible sets is another question, beyond the risk and contraction bounds proved here. In Gaussian sequence models, \citet[Theorem~3.6]{SzaboVanderVaartVanZanten2015} obtain coverage tending to one for inflated empirical Bayes credible balls under their polished-tail condition on the true coefficient sequence. In fixed-dimensional ABC, \citet{LiFearnhead2018a} establish asymptotically calibrated uncertainty for regression-adjusted posteriors under suitable assumptions. Analogous results for the nonparametric ABC models studied here would require control of estimation bias relative to posterior spread, as well as the effect of the ABC tolerance.

\section*{Declarations}
ChatGPT was used to assist with expositional editing, proof-checking support, and LaTeX formatting. The author verified the mathematical content and remains fully responsible for the correctness, originality, and integrity of the manuscript.

\appendix

\section{Auxiliary results}
\label{app:external}

\begin{lemma}[Weighted Gaussian-square concentration]
\label{lem:weighted-chi-square}
Let $(G_j)_{j\ge1}$ be independent standard Gaussian variables and let $(a_j)_{j\ge1}$ be a summable sequence of nonnegative weights, with finite sequences extended by zeros. Put $A_2=\left(\sum_{j=1}^{\infty}a_j^2\right)^{1/2}$ and $A_\infty=\sup_{j\ge1}a_j$. Then, for every $t>0$,
\begin{equation}
  \mathrm P\left\{\sum_{j=1}^{\infty}a_j(G_j^2-1)>2A_2\sqrt t+2A_\infty t\right\}\le e^{-t}.
  \label{eq:weighted-chi-square}
\end{equation}
\end{lemma}

\begin{proof}
The finite-sequence inequality is the weighted chi-square bound of \citet[Lemma 1]{LaurentMassart2000}. For the summable extension, if $A_2=0$, then all $a_j$ vanish and the bound holds. Assume $A_2>0$. For $m\in\mathbb N$, put $T_m=\sum_{j=1}^m a_j(G_j^2-1)$. For $0<s<(2A_\infty)^{-1}$, independence and $-\log(1-x)-x\le x^2\{2(1-x)\}^{-1}$ give
\begin{equation}
\begin{aligned}
 \log\mathrm E e^{sT_m}
 &=\frac12\sum_{j=1}^{m}\{-2sa_j-\log(1-2sa_j)\}\\
 &\le \frac{s^2\sum_{j=1}^m a_j^2}{1-2sA_\infty}
 \le \frac{s^2A_2^2}{1-2sA_\infty}.
\end{aligned}
\end{equation}
Chernoff's bound with $s=\sqrt t/(A_2+2A_\infty\sqrt t)$ yields
\begin{equation}
  \mathrm P\left\{T_m>2A_2\sqrt t+2A_\infty t\right\}\le e^{-t}.
\end{equation}
Since $\sum_{j=1}^{\infty} a_j\mathrm E|G_j^2-1|<\infty$, Tonelli's theorem shows that $T=\sum_{j=1}^{\infty}a_j(G_j^2-1)$ converges absolutely almost surely. Thus $T_m\to T$ almost surely. For every $\eta>0$, Fatou's lemma gives
\begin{equation}
  \mathrm P\left\{T>2A_2\sqrt t+2A_\infty t+\eta\right\}
  \le\liminf_{m\to\infty}\mathrm P\left\{T_m>2A_2\sqrt t+2A_\infty t\right\}
  \le e^{-t}.
\end{equation}
Letting $\eta\downarrow0$ proves \eqref{eq:weighted-chi-square}.
\end{proof}

For a centered real random variable $Y$, let $\lVert Y\rVert_{\psi_2}$ denote its sub-Gaussian Orlicz norm.

\begin{lemma}[Euclidean concentration from directional sub-Gaussianity]
\label{lem:directional-subgaussian}
Let $V$ be a centered random vector in $\mathbb R^K$ satisfying
\begin{equation}
  \sup_{u\in\mathbb S^{K-1}}\lVert u^\top V\rVert_{\psi_2}\le\sigma.
\end{equation}
Then there are universal $c,B_0>0$ such that, for every $B\ge B_0$,
\begin{equation}
  \mathrm P\left\{\lVert V\rVert_2>B\sigma\sqrt K\right\}\le e^{-cB^2K}.
  \label{eq:directional-subgaussian-vector}
\end{equation}
\end{lemma}

\begin{proof}
If $\sigma=0$, then every coordinate of $V$ has zero $\psi_2$ norm, so $V=0$ almost surely and the conclusion holds. Hence suppose $\sigma>0$. For each fixed $u$, the sub-Gaussian tail bound gives $\mathrm P\{|u^\top V|>t\}\le2e^{-ct^2/\sigma^2}$. Use a $1/2$-net of $\mathbb S^{K-1}$ with cardinality at most $5^K$ and the inequality $\lVert z\rVert_2\le2\max_{u\in\mathcal N}|u^\top z|$; compare \citet[Corollary~4.2.13 and Exercise~4.4.2]{Vershynin2018}. A union bound gives, for a universal $c_0>0$,
\begin{equation}
  \mathrm P\left\{\lVert V\rVert_2>B\sigma\sqrt K\right\}
  \le2\cdot5^K e^{-c_0B^2K}.
\end{equation}
Since $\log2+K\log5\le K\log10$, choosing $B_0$ so that $c_0B_0^2\ge2\log10$ proves the bound with $c=c_0/2$.
\end{proof}

\begin{corollary}[Central ball under a second-moment bound]
\label{cor:directional-small-ball}
Under the assumptions of Lemma~\ref{lem:directional-subgaussian}, if $\mathrm E\lVert V\rVert_2^2\le C_0K\sigma^2$ for some $C_0\ge0$, then there is $B_1=B_1(C_0)<\infty$ such that
\begin{equation}
  \mathrm P\left\{\lVert V\rVert_2\le B_1\sigma\sqrt K\right\}\ge\frac12.
\end{equation}
\end{corollary}

\begin{proof}
If $\sigma=0$, then $V=0$ almost surely. Otherwise, taking $B_1=\max\{1,\sqrt{2C_0}\}$ and applying Markov's inequality gives
\begin{equation}
  \mathrm P\left\{\lVert V\rVert_2>B_1\sigma\sqrt K\right\}
  \le \frac{\mathrm E\lVert V\rVert_2^2}{B_1^2\sigma^2K}
  \le\frac12.
\end{equation}
\end{proof}

\begin{lemma}[Bernstein--net bound]
\label{lem:bernstein-net}
Let $W_1,\ldots,W_m$ be iid centered vectors in $\mathbb R^K$. Suppose there are constants $H,V\ge0$ with $H+V>0$ such that, for every $u\in\mathbb S^{K-1}$,
\begin{equation}
  |u^\top W_1|\le H\sqrt K\quad\text{a.s.},
  \qquad
  \mathrm{Var}(u^\top W_1)\le V.
\end{equation}
Then there is a universal constant $c>0$ such that, for every $t>0$,
\begin{equation}
  \mathrm P\left\{\left\lVert m^{-1}\sum_{i=1}^mW_i\right\rVert_2>t\right\}
  \le2\cdot5^K\exp\left\{-\frac{cmt^2}{V+H\sqrt K\,t}\right\}.
  \label{eq:bernstein-net}
\end{equation}
\end{lemma}

\begin{proof}
Let $\mathcal N$ be a $1/2$-net of $\mathbb S^{K-1}$ with $|\mathcal N|\le5^K$. For every $z\in\mathbb R^K$, $\lVert z\rVert_2\le2\max_{u\in\mathcal N}|u^\top z|$. Hence the event on the left-hand side of \eqref{eq:bernstein-net} implies that, for some $u\in\mathcal N$,
\begin{equation}
  \left|m^{-1}\sum_{i=1}^m u^\top W_i\right|>t/2.
\end{equation}
For a fixed $u$, scalar Bernstein's inequality \citep[Theorem~2.8.4]{Vershynin2018}, the variance bound $\sum_{i=1}^{m}\mathrm{Var}(u^\top W_i)\le mV$, and the almost-sure bound $|u^\top W_i|\le H\sqrt K$ give
\begin{equation}
  \mathrm P\left\{\left|m^{-1}\sum_{i=1}^m u^\top W_i\right|>t/2\right\}
  \le2\exp\left\{-\frac{cmt^2}{V+H\sqrt K\,t}\right\}.
\end{equation}
A union bound over $\mathcal N$ proves \eqref{eq:bernstein-net}.
\end{proof}

\begin{theorem}[Sobolev algebra and composition]
\label{thm:sobolev-composition}
Let $-\infty<m<M<\infty$ and $\alpha>d/2$. Periodic $\mathcal H^\alpha([0,1]^d)$ is a Banach algebra. If $F$ is smooth on an open interval containing $[m,M]$, then $p\mapsto F\circ p$ maps bounded subsets of
\begin{equation}
  \{p\in \mathcal H^\alpha:m\le p\le M\}
\end{equation}
into bounded subsets of $\mathcal H^\alpha$. In particular, $p\mapsto\sqrt p$ has this property when $m>0$.
\end{theorem}

\begin{proof}
The algebra assertion follows from \citet[Theorem~4.6.4/1]{RunstSickel1996}, with $p=q=2$. For composition, use their Theorem~5.3.6/1(i) when $0<\alpha<1$ and Theorem~5.3.6/2(i) when $\alpha\ge1$, choosing the auxiliary smoothness index $\mu>\alpha$ in the latter. Let $I$ be an open interval containing $[m,M]$ on which $F$ is smooth. Since $[m,M]$ is compactly contained in $I$, choose a smooth cutoff $\chi$ with compact support in $I$ that equals one on a neighborhood of $[m,M]$. Extending $\chi F$ by zero outside $I$ gives a smooth function $\widetilde F$ on $\mathbb R$ with compact support and bounded derivatives of every order. The function $\widetilde F-\widetilde F(0)$ vanishes at zero, as required by the cited composition results. The periodic versions follow by localization; restoring the constant gives $F\circ p$ because $\widetilde F=F$ on $[m,M]$ and constants belong to periodic $\mathcal H^\alpha$.
\end{proof}

\begin{theorem}[Assouad hypercube bound]
\label{thm:external-assouad}
Let $\{\mathrm P_\nu:\nu\in\{-1,1\}^m\}$ index parameters $\vartheta_\nu$ such that
\begin{equation}
  d^2(\vartheta_\nu,\vartheta_{\nu'})\ge a^2H(\nu,\nu'),
\end{equation}
where $H$ is Hamming distance and $\nu^{(j)}$ denotes the vector obtained from $\nu$ by reversing its $j$th coordinate. If $\mathrm{TV}(\mathrm P_\nu,\mathrm P_{\nu^{(j)}})\le\eta<1$ for every neighboring pair, then
\begin{equation}
  \inf_{\widehat\vartheta}\sup_\nu\mathrm E_\nu d^2(\widehat\vartheta,\vartheta_\nu)
  \ge \frac{ma^2}{8}(1-\eta).
\end{equation}
\end{theorem}

\begin{proof}
This is Assouad's lemma in squared metric loss; see \citet[Theorem 2.12]{Tsybakov2009}.
\end{proof}

\begin{theorem}[Le Cam two-point lower bound]
\label{thm:external-lecam}
For two probability measures $\mathrm P_0,\mathrm P_1$ and parameter values $\theta_0,\theta_1$, every estimator satisfies
\begin{equation}
  \sup_{j=0,1}\mathrm E_jd^2(\widehat\theta,\theta_j)
  \ge\frac{d^2(\theta_0,\theta_1)}8\{1-\mathrm{TV}(\mathrm P_0,\mathrm P_1)\}.
\end{equation}
\end{theorem}

\begin{proof}
See \citet[Theorem 2.2]{Tsybakov2009}.
\end{proof}

\section{Proofs for the abstract results}
\label{app:meta-proofs}

\begin{proof}[Proof of Theorem~\ref{thm:meta-ideal}]
On $\mathcal E_n$, Conditions~\ref{ass:meta}(B2)--(B3) give
\begin{equation}
  D_n(X)\ge e^{-(C_\Pi+C_L)K_n}.
  \label{eq:proof-meta-denom}
\end{equation}
For $B_n(M)=\{f:d(f,f_0)>M\epsilon_n\}$, Conditions (B4)--(B5) and $L_n^{\mathrm{ABC}}\le1$ give
\begin{equation}
\begin{aligned}
 \int_{B_n(M)}L_n^{\mathrm{ABC}}(f;X)\Pi_n(\mathrm df)
 &\le e^{-C_D(M-C_S)^2K_n}+\Pi_n(\mathcal G_n^c)\\
 &\le e^{-C_D(M-C_S)^2K_n}
      +e^{-(C_\Pi+C_L+C_G)K_n}.
\end{aligned}
\end{equation}
After division by \eqref{eq:proof-meta-denom},
\begin{equation}
  \Pi_n^{\mathrm{ABC}}\{B_n(M)\mid X\}
  \le e^{-\{C_D(M-C_S)^2-C_\Pi-C_L\}K_n}+e^{-C_GK_n}.
\end{equation}
Choose $M_0\ge M_\star$ so that $C_D(M-C_S)^2-C_\Pi-C_L>0$ for every $M\ge M_0$, and set $\gamma_M=\min\{C_D(M-C_S)^2-C_\Pi-C_L,C_G\}$. On $\mathcal E_n^c$ the posterior probability is at most one. Adding $\mathrm P_{f_0}^{(n)}(\mathcal E_n^c)$ proves the expectation statement.
\end{proof}

\begin{proof}[Proof of Corollary~\ref{cor:meta-denom}]
The lower bound \eqref{eq:proof-meta-denom} holds on $\mathcal E_n$. Choose $\kappa>C_\Pi+C_L$ and $c'\in(0,c)$. Then $\mathrm P_{f_0}^{(n)}\{D_n(X)<e^{-\kappa K_n}\}\le\mathrm P_{f_0}^{(n)}(\mathcal E_n^c)\le e^{-cK_n}\le e^{-c'K_n}$ for all sufficiently large $n$.
\end{proof}

\begin{proof}[Proof of Theorem~\ref{thm:meta-mc}]
Conditional on $X$, the indicators are iid Bernoulli with success probability $D_n(X)$. Conditional on $X$ and $A_{n,N}=m\ge1$, the accepted parameters are iid with law $\Pi_n^{\mathrm{ABC}}(\cdot\mid X)$. Hence, for every measurable $B$,
\begin{equation}
  \mathrm E_{\mathrm{MC}}\{\widehat\Pi_{n,N}^{\mathrm{ABC}}(B\mid X)\mid X\}
  \le\Pi_n^{\mathrm{ABC}}(B\mid X)+(1-D_n(X))^{N_n}.
  \label{eq:proof-mc-mass}
\end{equation}
Corollary~\ref{cor:meta-denom} gives
\begin{equation}
  \sup_{f_0}\mathrm E_{f_0}^{(n)}(1-D_n(X))^{N_n}
  \le e^{-c'K_n}+\exp\{-N_n e^{-\kappa K_n}\}\to0.
\end{equation}
Apply Theorem~\ref{thm:meta-ideal} with $B=B_n(M)$.
\end{proof}

\begin{proof}[Proof of Corollary~\ref{cor:meta-mc-mean}]
Conditional on $X$ and $A_{n,N}=m\ge1$, let $F_1,\ldots,F_m$ be the accepted iid draws. Jensen's inequality gives
\begin{equation}
  \mathrm E_{\mathrm{MC}}\left\{d^2\left(m^{-1}\sum_{j=1}^mF_j,f_0\right)\mid X,A_{n,N}=m\right\}
  \le\int_{\Theta} d^2(f,f_0)\Pi_n^{\mathrm{ABC}}(\mathrm df\mid X).
\end{equation}
The assumed second moment gives Bochner integrability of the ideal posterior, while a finite empirical average is automatically defined on $A_{n,N}\ge1$. On $A_{n,N}=0$, the loss is bounded uniformly by \eqref{eq:meta-fallback-bounded}. Taking expectations, applying \eqref{eq:meta-ideal-second}, and using
\begin{equation}
  \sup_{f_0}\mathrm E_{f_0}^{(n)}\mathrm P_{\mathrm{MC}}(A_{n,N}=0\mid X)
  \le e^{-c'K_n}+\exp\{-N_n e^{-\kappa K_n}\}
\end{equation}
proves the result under \eqref{eq:meta-N-risk}.
\end{proof}

\section{Proofs for orthogonal-series regression with \texorpdfstring{$g$-and-$k$}{g-and-k} errors}
\label{app:gk-proofs}

\begin{lemma}[Sobolev projection bias]
\label{lem:reg-tail}
For $\theta_0\in\Theta_\alpha(R)$,
\begin{equation}
  \lVert\bm\theta_{0,>K}\rVert_{\ell^2}\le RK^{-\alpha}.
\end{equation}
\end{lemma}

\begin{proof}
For $j>K$, $j^{2\alpha}\ge K^{2\alpha}$, and the defining ellipsoid bound applies.
\end{proof}

\begin{lemma}[Fixed Gaussian prior local mass]
\label{lem:reg-prior-mass}
There are $B,\delta,C>0$ such that, uniformly over $\theta_0\in\Theta_\alpha(R)$,
\begin{equation}
  \Pi\left\{\lVert\bm\theta_K-\bm\theta_{0,K}\rVert_2\le\delta\epsilon_n,
  \ \lVert\bm\theta_{>K}\rVert_{\ell^2}\le B\epsilon_n\right\}
  \ge e^{-CK}
  \label{eq:reg-prior-local-mass}
\end{equation}
for $K=K_n$ and all sufficiently large $n$.
\end{lemma}

\begin{proof}
Condition on $\Lambda$, and write $\mathrm P_\Lambda$ and $\mathrm E_\Lambda$ for the corresponding conditional probability and expectation. The first $K$ coordinates have standard deviations $\sigma_j=\Lambda j^{-\alpha-1/2}$. If $u=\bm\theta_{0,K}+h$ and $\lVert h\rVert_2\le\delta\epsilon_n$, then, uniformly in $\Lambda\in[\lambda_-,\lambda_+]$,
\begin{equation}
\begin{aligned}
  \sum_{j=1}^K\frac{u_j^2}{\sigma_j^2}
  &\le \frac{2}{\lambda_-^2}\sum_{j=1}^K j^{2\alpha+1}\theta_{0j}^2
     +\frac{2}{\lambda_-^2}\sum_{j=1}^K j^{2\alpha+1}h_j^2\\
  &\le \frac{2KR^2}{\lambda_-^2}
     +\frac{2K^{2\alpha+1}\delta^2\epsilon_n^2}{\lambda_-^2}
  \le CK.
\end{aligned}
\end{equation}
Here the last inequality uses $K^{2\alpha+1}\epsilon_n^2=O(K)$. The Gaussian density lower bound, the $K$-ball volume formula, and Stirling's inequality give
\begin{equation}
\begin{aligned}
 \log\mathrm P_\Lambda\{\lVert\bm\theta_K-\bm\theta_{0,K}\rVert_2\le\delta\epsilon_n\}
 &\ge -CK+\left(\alpha+\frac12\right)\sum_{j=1}^K\log j\\
 &\quad+K\log(\delta\epsilon_n)-\frac K2\log K-C K.
\end{aligned}
\end{equation}
Since $\epsilon_n\asymp K^{-\alpha}$ and $\sum_{j=1}^K\log j=K\log K-K+O(\log K)$, all $K\log K$ terms cancel, leaving a lower bound $-CK$, uniformly in $\Lambda\in[\lambda_-,\lambda_+]$.

The tail is independent of the first $K$ coordinates conditional on $\Lambda$, and
\begin{equation}
  \mathrm E_\Lambda\lVert\bm\theta_{>K}\rVert_{\ell^2}^2
  =\Lambda^2\sum_{j=K+1}^{\infty}j^{-2\alpha-1}\le CK^{-2\alpha}\le C\epsilon_n^2.
\end{equation}
Choosing $B$ large, Markov's inequality gives $\mathrm P_\Lambda\{\lVert\bm\theta_{>K}\rVert_{\ell^2}\le B\epsilon_n\}\ge1/2$. Integrating over the scale hyperprior proves the claim.
\end{proof}

\begin{lemma}[Exponentially likely Gaussian prior tail sieves]
\label{lem:reg-prior-sieve}
For every $L>0$, a constant $B_1<\infty$ can be chosen so that
\begin{equation}
  \mathcal G_n=\left\{\theta:\lVert\bm\theta_{>K_n}\rVert_{\ell^2}\le B_1\epsilon_n\right\}
  \label{eq:reg-prior-sieve}
\end{equation}
satisfies
\begin{equation}
  \Pi(\mathcal G_n^c)\le e^{-LK_n}
  \label{eq:reg-prior-sieve-tail}
\end{equation}
for all sufficiently large $n$.
\end{lemma}

\begin{proof}
Put $K=K_n$ and condition on $\Lambda$. Apply Lemma~\ref{lem:weighted-chi-square} to the tail weights $a_j=\Lambda^2j^{-2\alpha-1}$, $j>K$. Their sum, $\ell^2$ norm, and maximum are bounded by constant multiples of $K^{-2\alpha}$, $K^{-2\alpha-1/2}$, and $K^{-2\alpha-1}$, uniformly in $\Lambda\in[\lambda_-,\lambda_+]$. Taking $t=LK$ gives
\begin{equation}
  \mathrm P_\Lambda\left\{\lVert\bm\theta_{>K}\rVert_{\ell^2}^2>
  C(1+\sqrt L+L)K^{-2\alpha}\right\}\le e^{-LK}.
\end{equation}
Since $\epsilon_n\ge K^{-\alpha}$, choose $B_1^2\ge C(1+\sqrt L+L)$ and integrate over the scale hyperprior.
\end{proof}

\begin{proof}[Proof of Proposition~\ref{prop:implicit-error}]
Write $t=gz/2$ and $r(t)=\tanh t+t\operatorname{sech}^2t$. The bounds $|\tanh t|\le1$ and $|t|\operatorname{sech}^2t\le1/2$ give $|r(t)|\le3/2$. Consequently,
\begin{equation}
  h_{g,c}'(z)=1+c r(gz/2)\in[1-3c/2,1+3c/2]\subset[1/4,7/4].
\end{equation}
Thus $h_{g,c}$ is a smooth increasing bijection. Its second and third derivatives are bounded because derivatives of $\tanh$ are bounded and the terms containing $t$ are multiplied by powers of $\operatorname{sech}t$. The definition of $\mu_{g,c}$ centers $h_{g,c}(G)$. Since $h_{g,c}$ is Lipschitz, Gaussian concentration \citep[Theorem~5.2.2]{Vershynin2018} implies that $h_{g,c}(G)$ is sub-Gaussian. The change-of-variables formula gives \eqref{eq:implicit-error-density}.

Let $\ell=\log q_{g,c}$. Writing $x=h_{g,c}(z)$,
\begin{equation}
  \ell'(x)=\frac{-z-h_{g,c}''(z)/h_{g,c}'(z)}{h_{g,c}'(z)}.
\end{equation}
Writing $a_1=h_{g,c}'$, $a_2=h_{g,c}''$, and $a_3=h_{g,c}'''$, differentiation gives
\begin{equation}
  \ell''\{h_{g,c}(z)\}
  =-a_1(z)^{-2}+z a_2(z)a_1(z)^{-3}-a_3(z)a_1(z)^{-3}+2a_2(z)^2a_1(z)^{-4}.
\end{equation}
The derivative lower bound and the boundedness of $a_2$ and $a_3$ therefore imply
\begin{equation}
  |\ell''\{h_{g,c}(z)\}|\le C(1+|z|).
  \label{eq:implicit-log-density-second}
\end{equation}
Moreover, $|h_{g,c}^{-1}(h_{g,c}(z)+t)-z|\le4|t|$. Taylor's theorem and \eqref{eq:implicit-log-density-second} therefore give, uniformly for $|t|\le t_0$,
\begin{equation}
\begin{aligned}
 \mathrm{KL}\{q_{g,c},q_{g,c}(\cdot-t)\}
 &=\int_{\mathbb R} q_{g,c}(x)\{\ell(x)-\ell(x-t)\}\,\mathrm dx\\
 &\le |t|\left|\int_{\mathbb R} q_{g,c}(x)\ell'(x)\,\mathrm dx\right|+Ct^2\mathrm E(1+|G|).
\end{aligned}
\end{equation}
The preceding bounds imply that $q_{g,c}(x)\ell'(x)=q_{g,c}'(x)$ is integrable. Since $q_{g,c}(x)\to0$ as $x\to\pm\infty$,
\begin{equation}
  \int_{\mathbb R}q_{g,c}(x)\ell'(x)\,\mathrm dx
  =\int_{\mathbb R}q_{g,c}'(x)\,\mathrm dx=0.
\end{equation}
Translation invariance yields \eqref{eq:implicit-shift-kl}.
\end{proof}

\begin{lemma}[Orthogonal-regression summary concentration]
\label{lem:implicit-summary}
For $1\le K\le n$, define $S_{n,K}$ by \eqref{eq:implicit-reg-summary} with $K_n$ replaced by $K$. For the observed or synthetic regression summary, constants $c,B_0>0$ satisfy
\begin{equation}
  \mathrm P\left\{\lVert S_{n,K}(X)-\bm\theta_K\rVert_2>B\sqrt{K/n}\right\}\le e^{-cB^2K},
  \qquad B\ge B_0,
  \label{eq:implicit-summary-concentration}
\end{equation}
and
\begin{equation}
  \mathrm P\left\{\lVert S_{n,K}(X)-\bm\theta_K\rVert_2\le B_0\sqrt{K/n}\right\}\ge\frac12.
  \label{eq:implicit-summary-small-ball}
\end{equation}
\end{lemma}

\begin{proof}
Put $a_i=(\phi_{1,n}(i),\ldots,\phi_{K,n}(i))^\top$ and $V=n^{-1}\sum_{i=1}^n a_i\varepsilon_i$. For $u\in\mathbb S^{K-1}$, orthogonality gives
\begin{equation}
  \sum_{i=1}^n(u^\top a_i)^2=n.
\end{equation}
The independent sub-Gaussian sum inequality \citep[Proposition~2.6.1]{Vershynin2018} therefore yields $\lVert u^\top V\rVert_{\psi_2}\le Cn^{-1/2}$ uniformly in $u$. Lemma~\ref{lem:directional-subgaussian}, with $\sigma=Cn^{-1/2}$, proves \eqref{eq:implicit-summary-concentration}. Orthogonality also gives $\mathrm E\lVert V\rVert_2^2=\mathrm{Var}(\varepsilon)K/n$, so Markov's inequality proves \eqref{eq:implicit-summary-small-ball} after increasing $B_0$.
\end{proof}

\begin{proof}[Proof of Theorem~\ref{thm:implicit-reg-contract}]
Let $K=K_n$, fix $B_X\ge B_0$ as in Lemma~\ref{lem:implicit-summary}, and define
\begin{equation}
  \mathcal E_n=\left\{\lVert S_{n,K}(X)-\bm\theta_{0,K}\rVert_2
  \le B_X\sqrt{K/n}\right\}.
\end{equation}
Lemma~\ref{lem:implicit-summary} gives Condition~\ref{ass:meta}(B1). Use the local set in Lemma~\ref{lem:reg-prior-mass}. Its prior mass is at least $e^{-C_\Pi K}$, and Lemma~\ref{lem:reg-tail} gives
\begin{equation}
  \lVert\theta-\theta_0\rVert_{\ell^2}
  \le\delta\epsilon_n+B\epsilon_n+RK^{-\alpha}
  \le(\delta+B+R)\epsilon_n
\end{equation}
throughout that set. On $\mathcal E_n$, the synthetic event
\begin{equation}
  \lVert S_{n,K}(Z)-\bm\theta_K\rVert_2\le B_0\sqrt{K/n}
\end{equation}
implies acceptance whenever $A>B_X+\delta+B_0$. Its probability is at least $1/2$, uniformly over the local set. Thus (B2)--(B3) hold with a sufficiently large constant $C_L$.

Choose the tail sieve of Lemma~\ref{lem:reg-prior-sieve} with $L>C_\Pi+C_L+2$. For $\theta\in\mathcal G_n$ such that $\lVert\theta-\theta_0\rVert_{\ell^2}>M\epsilon_n$, the head coordinates satisfy
\begin{equation}
  \lVert\bm\theta_K-\bm\theta_{0,K}\rVert_2
  >(M-B_1-R)\epsilon_n.
\end{equation}
On $\mathcal E_n$, acceptance therefore forces
\begin{equation}
  \lVert S_{n,K}(Z)-\bm\theta_K\rVert_2
  >(M-B_1-R-B_X-A)\epsilon_n.
\end{equation}
For every sufficiently large fixed $M$, Lemma~\ref{lem:implicit-summary} and $\epsilon_n\ge\sqrt{K/n}$ bound this probability by $\exp\{-c(M-B_1-R-B_X-A)^2K\}$, which is (B4). The tail-sieve probability is (B5). Theorem~\ref{thm:meta-ideal}, together with $\epsilon_n\asymp\rho_{n,\alpha}$, proves the result.
\end{proof}

\begin{proof}[Proof of Corollary~\ref{cor:implicit-reg-risk}]
For all sufficiently large $n$, the target class lies in the ball of radius $r_n$. Metric projection onto this ball is nonexpansive, so $\lVert \mathcal C_n(\theta)-\theta_0\rVert_{\ell^2}\le\lVert\theta-\theta_0\rVert_{\ell^2}$. Fix $M\ge M_0$. By \eqref{eq:meta-ideal-on-event} and the proof of Theorem~\ref{thm:implicit-reg-contract}, on $\mathcal E_n$,
\begin{equation}
\begin{aligned}
 \int_{\ell^2}\lVert \mathcal C_n(\theta)-\theta_0\rVert_{\ell^2}^2\Pi_{n,\mathrm{reg}}^{\mathrm{ABC}}(\mathrm d\theta\mid X)
 &\le M^2\epsilon_n^2+(r_n+R)^2\,2e^{-\gamma_MK_n}.
\end{aligned}
\end{equation}
On the complement, the integral is at most $(r_n+R)^2$ and the complement probability is at most $e^{-cK_n}$. Since $r_n^2=K_n$ and polynomial factors are dominated by the exponential terms, expectation is $O(\epsilon_n^2)$. Jensen's inequality proves the result.
\end{proof}

\begin{proof}[Proof of Theorem~\ref{thm:implicit-reg-lower}]
Let $K=\lfloor c_K n^{1/(2\alpha+1)}\rfloor$ and use the coordinate block $\mathcal J=\{K+1,\ldots,2K\}$, which is contained in $\{1,\ldots,n\}$ for large $n$. For $\nu\in\{-1,1\}^{\mathcal J}$, set
\begin{equation}
  \theta_{\nu,j}=\begin{cases}a_n\nu_j,&j\in\mathcal J,\\0,&j\notin\mathcal J,\end{cases}
  \qquad a_n=c_a n^{-1/2}.
\end{equation}
Choose fixed $c_K,c_a>0$ with $2^{2\alpha}c_a^2c_K^{2\alpha+1}\le R^2$. Then every vertex belongs to $\Theta_\alpha(R)$, because
\begin{equation}
\begin{aligned}
  \sum_{j=1}^{\infty}j^{2\alpha}\theta_{\nu,j}^2
  &\le (2K)^{2\alpha}K a_n^2
   =2^{2\alpha}c_a^2K^{2\alpha+1}/n\\
  &\le 2^{2\alpha}c_a^2c_K^{2\alpha+1}\le R^2.
\end{aligned}
\end{equation}
For all vertices, the coordinate construction gives
\begin{equation}
  \lVert\theta_\nu-\theta_{\nu'}\rVert_{\ell^2}^2
  =4a_n^2H(\nu,\nu'),
\end{equation}
where $H$ is Hamming distance. Neighboring vertices therefore have squared $\ell^2$ distance $4a_n^2$. Their regression mean vectors differ by $2a_n\{\phi_{j,n}(i)\}_{i=1}^n$. Since $|\phi_{j,n}(i)|\le\sqrt2$, the individual shifts lie in the range of Proposition~\ref{prop:implicit-error} for large $n$, and
\begin{equation}
  \mathrm{KL}(\mathrm P_\nu^{(n)},\mathrm P_{\nu^{(j)}}^{(n)})
  \le C_q\sum_{i=1}^n4a_n^2\phi_{j,n}(i)^2
  =4C_qna_n^2.
\end{equation}
Reducing $c_a$ if necessary, take $c_a\le(8C_q)^{-1/2}$; this preserves the ellipsoid inclusion. Pinsker's inequality now gives, for every neighboring pair,
\begin{equation}
\begin{aligned}
  \mathrm{TV}(\mathrm P_\nu^{(n)},\mathrm P_{\nu^{(j)}}^{(n)})
  &\le\left\{\frac12\mathrm{KL}(\mathrm P_\nu^{(n)},\mathrm P_{\nu^{(j)}}^{(n)})\right\}^{1/2}\\
  &\le\sqrt{2C_qn a_n^2}
   =\sqrt{2C_q}\,c_a\le\frac12.
\end{aligned}
\end{equation}
Theorem~\ref{thm:external-assouad} applies with hypercube dimension $K$, squared separation parameter $4a_n^2$, and $\eta=1/2$. Since all vertices lie in the target ellipsoid,
\begin{equation}
\begin{aligned}
  &\inf_{\widetilde\theta_n}\sup_{\theta_0\in\Theta_\alpha(R)}
     \mathrm E_{\theta_0}^{(n)}
     \lVert\widetilde\theta_n-\theta_0\rVert_{\ell^2}^2\\
  &\quad\ge\inf_{\widetilde\theta_n}\max_{\nu\in\{-1,1\}^{\mathcal J}}
     \mathrm E_\nu^{(n)}
     \lVert\widetilde\theta_n-\theta_\nu\rVert_{\ell^2}^2\\
  &\quad\ge\frac{K(4a_n^2)}8\left(1-\frac12\right)
   =\frac{c_a^2K}{4n}.
\end{aligned}
\end{equation}
Both infima are over measurable estimators based on the regression sample. For all sufficiently large $n$, $c_Kn^{1/(2\alpha+1)}\ge2$, so
\begin{equation}
  K=\left\lfloor c_Kn^{1/(2\alpha+1)}\right\rfloor
  \ge\frac{c_K}{2}n^{1/(2\alpha+1)}.
\end{equation}
Consequently,
\begin{equation}
  \inf_{\widetilde\theta_n}\sup_{\theta_0\in\Theta_\alpha(R)}
  \mathrm E_{\theta_0}^{(n)}
  \lVert\widetilde\theta_n-\theta_0\rVert_{\ell^2}^2
  \ge\frac{c_a^2c_K}{8}n^{-2\alpha/(2\alpha+1)},
\end{equation}
which proves the stated lower bound.
\end{proof}

\begin{proof}[Proof of Corollary~\ref{cor:implicit-reg-mc}]
The posterior-mass assertion follows from Theorem~\ref{thm:meta-mc}.
\end{proof}

\begin{proof}[Proof of Corollary~\ref{cor:implicit-reg-mc-risk}]
Condition on the observed data $X$. Given $A_{n,N}=m\ge1$, the $m$ accepted parameter draws are independent with common law $\Pi_{n,\mathrm{reg}}^{\mathrm{ABC}}(\cdot\mid X)$. Since $\mathcal C_n$ is measurable, Jensen's inequality gives
\begin{equation}
\begin{aligned}
 &\mathrm E_{\mathrm{MC}}\left[\left.\left\lVert
   m^{-1}\sum_{\substack{1\le i\le N_n\\ I_i=1}}\mathcal C_n(\theta_i)-\theta_0
   \right\rVert_{\ell^2}^2\,\right|X,A_{n,N}=m\right]\\
 &\qquad\le
 \int_{\ell^2}\lVert \mathcal C_n(\theta)-\theta_0\rVert_{\ell^2}^2
 \Pi_{n,\mathrm{reg}}^{\mathrm{ABC}}(\mathrm d\theta\mid X).
\end{aligned}
\end{equation}
The proof of Corollary~\ref{cor:implicit-reg-risk} shows that the expectation of the right-hand side, uniformly over $\theta_0\in\Theta_\alpha(R)$, is $O(\epsilon_n^2)$. On $\{A_{n,N}=0\}$, the empirical estimator equals the fallback value $0$, and hence its squared loss is at most $R^2$. Corollary~\ref{cor:meta-denom} and the conditional binomial formula yield
\begin{equation}
\begin{aligned}
 \sup_{\theta_0\in\Theta_\alpha(R)}
 \mathrm E_{\theta_0}^{(n)}
 \mathrm P_{\mathrm{MC}}(A_{n,N}=0\mid X)
 &\le e^{-c'K_n}+\exp\left\{-N_n e^{-\kappa K_n}\right\}.
\end{aligned}
\end{equation}
Because $K_n\asymp n\epsilon_n^2$ diverges polynomially, the first term is $O(\epsilon_n^2)$. The assumed lower bound on $N_n e^{-\kappa K_n}$, with $C_N$ sufficiently large, makes the second term $O(\epsilon_n^2)$. Combining the positive-acceptance and no-acceptance contributions proves the claim.
\end{proof}

\section{Proofs for nonparametric decompounding}
\label{app:density-proofs}

\begin{lemma}[Trigonometric projection bounds]
\label{lem:density-projection}
Uniformly over $p_0\in\mathcal P_\alpha(R,m,M)$, the function $g_0$ in \eqref{eq:density-target-square-root} satisfies, for every integer $J\ge1$,
\begin{equation}
  \lVert g_0\rVert_{\mathcal H^\alpha}\le C,
  \qquad
  \lVert g_0-g_{0,J}\rVert_2\le CJ^{-\alpha},
  \qquad
  \lVert g_0-g_{0,J}\rVert_\infty\le CJ^{d/2-\alpha},
  \label{eq:density-projection-bounds}
\end{equation}
and $\sup_J\lVert g_{0,J}\rVert_\infty\le C$.
\end{lemma}

\begin{proof}
The bounds $m\le p_0\le M$ imply $\sqrt m\le\sqrt{p_0}\le\sqrt M$ and $\sqrt m\le c_0\le1$. Theorem~\ref{thm:sobolev-composition} gives the uniform $\mathcal H^\alpha$ bound for $g_0$. The $\mathcal L^2$ projection tail follows from the weighted coefficient bound. For the supremum norm, the Cauchy--Schwarz inequality gives
\begin{equation}
 \sup_{x\in\mathcal X}\sum_{\{j\ge2:q_j>J\}}|\theta_j(g_0)\varphi_j(x)|
 \le C\left\{\sum_{\{j\ge2:q_j>J\}}(1+q_j)^{-2\alpha}\right\}^{1/2}
 \le CJ^{d/2-\alpha},
\end{equation}
because the number of product basis functions in block $q$ is $O(q^{d-1})$. The same calculation without the tail restriction gives the uniform bound for the partial projections.
\end{proof}

\begin{lemma}[Local Lipschitz property of the normalized-square map]
\label{lem:density-square-map}
For $g\in\mathcal L^4(\mathcal X)\setminus\{0\}$, let $T(g)=g^2/\lVert g\rVert_2^2$. There are $C,\delta>0$ such that, for all sufficiently large $n$, for $J=J_n$, and for every $h\in\mathcal V_J$ with $\lVert h-g_{0,J}\rVert_2\le\delta\epsilon_n$,
\begin{equation}
  \lVert T(h)-p_0\rVert_2\le C\epsilon_n,
  \label{eq:density-local-map}
\end{equation}
uniformly over the target class.
\end{lemma}

\begin{proof}
The constant coefficient of both $g_0$ and $g_{0,J}$ is one, so their $\mathcal L^2$ norms are at least one. Lemma~\ref{lem:density-projection} gives uniform $\mathcal L^\infty$ bounds for $g_{0,J}$. For $h-g_{0,J}\in\mathcal V_J$, the identity $\sum_{j=1}^{D_J}\varphi_j(x)^2=D_J$ gives
\begin{equation}
  \lVert h-g_{0,J}\rVert_\infty\le\sqrt{D_J}\lVert h-g_{0,J}\rVert_2
  \le C J^{d/2-\alpha},
\end{equation}
which tends to zero. Also $\lVert h\rVert_2\ge\lVert g_{0,J}\rVert_2-\delta\epsilon_n\ge1/2$ for large $n$. For any $h$ and $g$ whose $\mathcal L^\infty$ norms are uniformly bounded and whose $\mathcal L^2$ norms are bounded away from zero,
\begin{equation}
\begin{aligned}
  T(h)-T(g)
  &=\frac{(h-g)(h+g)}{\lVert h\rVert_2^2}
    +g^2\frac{\lVert g\rVert_2^2-\lVert h\rVert_2^2}
    {\lVert h\rVert_2^2\lVert g\rVert_2^2}.
\end{aligned}
\end{equation}
The uniform $\mathcal L^\infty$ bounds, the lower bounds on the $\mathcal L^2$ norms, and
$|\lVert h\rVert_2^2-\lVert g\rVert_2^2|\le(\lVert h\rVert_2+\lVert g\rVert_2)\lVert h-g\rVert_2$ therefore give
\begin{equation}
  \lVert T(h)-T(g)\rVert_2\le C\lVert h-g\rVert_2.
\end{equation}
Apply this first to $h,g_{0,J}$ and then to $g_{0,J},g_0$, using Lemma~\ref{lem:density-projection} and $T(g_0)=p_0$.
\end{proof}

\begin{proof}[Proof of Proposition~\ref{prop:density-simulator}]
Put $J=J_n$. For the full product trigonometric block,
\begin{equation}
  \sum_{j=1}^{D_J}\varphi_j(x)^2
  =\prod_{r=1}^d\left[1+2\sum_{q=1}^J\{\cos^2(2\pi qx_r)+\sin^2(2\pi qx_r)\}\right]
  =(2J+1)^d=D_J.
\end{equation}
The Cauchy--Schwarz inequality gives $g_v(x)^2\le D_J\lVert g_v\rVert_2^2$, proving \eqref{eq:density-envelope}. Let $U\sim\mathrm{Unif}([0,1]^d)$ and $V\sim\mathrm{Unif}(0,1)$ be independent. The draw $U$ is accepted when $V\le p_v(U)/D_J$. For every measurable $B\subset[0,1]^d$,
\begin{equation}
  \mathrm P\{U\in B,\ V\le p_v(U)/D_J\}
  =D_J^{-1}\int_B p_v(x)\,\mathrm dx.
\end{equation}
The total acceptance probability is therefore $D_J^{-1}$, and the conditional law of an accepted draw has density $p_v$.
\end{proof}

\begin{proof}[Proof of Proposition~\ref{prop:density-local-mass}]
Put $J=J_n$. The approximation assertion follows from Lemmas~\ref{lem:density-projection} and~\ref{lem:density-square-map}. It remains to prove the coefficient prior-mass bound. Let $\widetilde K=D_J-1$ and let $\sigma_j=\Lambda(1+q_j)^{-\alpha-d/2}$. The Gaussian density on the ball $B_{\widetilde K}(v_{0,J},\delta\epsilon_n)$ is bounded below by
\begin{equation}
  (2\pi)^{-\widetilde K/2}\prod_{j=2}^{D_J}\sigma_j^{-1}
  \exp\left\{-\frac12\sup_{\lVert u-v_{0,J}\rVert_2\le\delta\epsilon_n}\sum_{j=2}^{D_J}u_j^2/\sigma_j^2\right\}.
\end{equation}
The Sobolev bound for $g_0$ gives
\begin{equation}
  \sum_{j=2}^{D_J}\frac{v_{0,J,j}^2}{\sigma_j^2}
  \le C J^d\sum_{j=2}^{D_J}(1+q_j)^{2\alpha}v_{0,J,j}^2\le C\widetilde K.
\end{equation}
For $\lVert u-v_{0,J}\rVert_2\le\delta\epsilon_n$, the coefficient perturbation satisfies $\sum_{j=2}^{D_J}(u_j-v_{0,J,j})^2/\sigma_j^2\le CJ^{2\alpha+d}\epsilon_n^2\le C\widetilde K$. For $q=1,\ldots,J$, the block $q$ contains
\begin{equation}
  N_q=(2q+1)^d-(2q-1)^d
\end{equation}
basis functions. Hence, by an integral comparison,
\begin{equation}
\begin{aligned}
  \sum_{j=2}^{D_J}\log(1+q_j)
  &=\sum_{q=1}^{J}N_q\log(1+q)
    =\widetilde K\log J+O(\widetilde K),\\
  \log \widetilde K&=d\log J+O(1).
\end{aligned}
\end{equation}
The logarithm of the Gaussian normalizing factor contributes $(\alpha+d/2)\widetilde K\log J+O(\widetilde K)$, while the logarithm of the ball volume contributes
\begin{equation}
  \widetilde K\log(\delta\epsilon_n)-\frac{\widetilde K}{2}\log \widetilde K+O(\widetilde K)
  =-(\alpha+d/2)\widetilde K\log J+O(\widetilde K).
\end{equation}
The logarithmic terms cancel, leaving a conditional probability at least $e^{-C\widetilde K}$ uniformly in $\Lambda$. Integration over the scale hyperprior gives \eqref{eq:density-square-root-local-mass}. Lemma~\ref{lem:density-square-map} supplies the local-distance assertion. The coefficient ball is compact, and the map $v\mapsto p_v$ is continuous from the finite-dimensional coefficient space to $\mathcal L^2$; its image is therefore compact and Borel. The pushforward measure of this image is at least the $\Gamma_n$ probability of the coefficient ball. Finally, $\widetilde K\asymp K_n$.
\end{proof}

\begin{proof}[Proof of Proposition~\ref{prop:decompound-stability}]
For $|z|,|w|\le1$, the fundamental theorem of calculus gives
\begin{equation}
  F_\beta(z)-F_\beta(w)
  =\beta e^{-\beta}(z-w)\int_0^1 e^{\beta\{w+t(z-w)\}}\,\mathrm dt.
  \label{eq:decompound-complex-difference}
\end{equation}
The line segment from $w$ to $z$ lies in the closed unit disk $\{\zeta\in\mathbb C:|\zeta|\le1\}$. Since $0<\beta<\pi/2$, for every $t\in[0,1]$,
\begin{equation}
\begin{aligned}
  \operatorname{Re}e^{\beta\{w+t(z-w)\}}
  &\ge e^{-\beta}\cos\beta>0,\\
  \left|e^{\beta\{w+t(z-w)\}}\right|&\le e^\beta.
\end{aligned}
\end{equation}
The modulus of the integral in \eqref{eq:decompound-complex-difference} is therefore between $e^{-\beta}\cos\beta$ and $e^\beta$. Hence
\begin{equation}
  c_\beta|z-w|\le|F_\beta(z)-F_\beta(w)|\le\beta|z-w|.
  \label{eq:decompound-complex-stability}
\end{equation}
Each $a_{\boldsymbol k}(p)$ has modulus at most one. For each one-dimensional frequency $q\ge1$,
\begin{equation*}
  e^{\pm2\pi\mathbf i qx}
  =\frac{t_{2q-1}(x)\pm\mathbf i\,t_{2q}(x)}{\sqrt2}.
\end{equation*}
This orthonormal change of basis replaces each sine--cosine pair by the positive- and negative-frequency exponentials and leaves the constant unchanged. Applying it in each coordinate shows that the product trigonometric basis and the complex exponential basis are orthonormal bases of the space obtained from $\mathcal V_{2J}$ by allowing complex coefficients. Changing between them preserves the sum of squared coefficient moduli. The constant coefficients cancel, because both densities integrate to one and both continuous aggregate components have mass $1-e^{-\beta}$. Thus
\begin{equation}
\begin{aligned}
  \lVert s_J(\beta,p)-s_J(\beta,\widetilde p)\rVert_2^2
  &=\sum_{\substack{\boldsymbol k\in\mathbb Z^d\\0<\lVert\boldsymbol k\rVert_\infty\le2J}}
  \left|F_\beta\{a_{\boldsymbol k}(p)\}-F_\beta\{a_{\boldsymbol k}(\widetilde p)\}\right|^2,\\
  \lVert \mathrm{Proj}_{2J}(p-\widetilde p)\rVert_2^2
  &=\sum_{\substack{\boldsymbol k\in\mathbb Z^d\\0<\lVert\boldsymbol k\rVert_\infty\le2J}}
  |a_{\boldsymbol k}(p)-a_{\boldsymbol k}(\widetilde p)|^2.
\end{aligned}
\end{equation}
Apply \eqref{eq:decompound-complex-stability} term by term and take square roots.
\end{proof}

\begin{proof}[Proof of Proposition~\ref{prop:decompound-joint-stability}]
For $|z|\le1$ and $\beta\in \mathcal I$, differentiation with respect to the real parameter $\beta$ gives
\begin{equation}
  \partial_\beta F_\beta(z)
  =e^{-\beta}\left\{z e^{\beta z}-(e^{\beta z}-1)\right\}.
\end{equation}
Using $e^{\beta z}-1=\beta z\int_0^1e^{t\beta z}\,\mathrm dt$ and $|e^{t\beta z}|\le e^{t\beta}$, we obtain
\begin{equation}
  |\partial_\beta F_\beta(z)|\le(1+\overline\beta)|z|.
  \label{eq:decompound-intensity-derivative}
\end{equation}
Integrate this bound between $\beta_0$ and $\beta$. The orthonormal change of basis used in the proof of Proposition~\ref{prop:decompound-stability} and Parseval's identity imply
\begin{equation}
\begin{aligned}
  \lVert s_J(\beta,p_0)-s_J(\beta_0,p_0)\rVert_2
  &\le (1+\overline\beta)|\beta-\beta_0|
  \left(\sum_{\substack{\boldsymbol k\in\mathbb Z^d\\0<\lVert\boldsymbol k\rVert_\infty\le2J}}
           |a_{\boldsymbol k}(p_0)|^2\right)^{1/2}\\
  &\le (1+\overline\beta)\lVert p_0-1\rVert_2|\beta-\beta_0|\\
  &\le C_M|\beta-\beta_0|,
\end{aligned}
\label{eq:decompound-intensity-sensitivity}
\end{equation}
where $C_M=(1+\overline\beta)\sqrt{M-1}$, since $\lVert p_0-1\rVert_2^2=\int_{\mathcal X} p_0(x)^2\,\mathrm dx-1\le M-1$. This comparison is at the bounded density $p_0$ and does not require a bound on $\lVert p\rVert_\infty$.

The zero-count coordinate satisfies
\begin{equation}
  e^{-\overline\beta}|\beta-\beta_0|
  \le|e^{-\beta}-e^{-\beta_0}|
  \le|\beta-\beta_0|.
  \label{eq:decompound-zero-stability}
\end{equation}
Also $c_\beta\ge c_{\mathcal I}$, where $c_{\mathcal I}=\underline\beta e^{-2\overline\beta}\cos\overline\beta>0$. Proposition~\ref{prop:decompound-stability}, applied at the common intensity $\beta$, and \eqref{eq:decompound-intensity-sensitivity} give
\begin{equation}
\begin{aligned}
 c_{\mathcal I}\lVert\mathrm{Proj}_{2J}(p-p_0)\rVert_2
 &\le\lVert s_J(\beta,p)-s_J(\beta_0,p_0)\rVert_2
       +C_M|\beta-\beta_0|\\
 &\le\left(1+C_M e^{\overline\beta}\right)
       \lVert t_J(\beta,p)-t_J(\beta_0,p_0)\rVert_2.
\end{aligned}
\end{equation}
Together with \eqref{eq:decompound-zero-stability}, this bounds
$|\beta-\beta_0|+\lVert\mathrm{Proj}_{2J}(p-p_0)\rVert_2$
by $\{e^{\overline\beta}+c_{\mathcal I}^{-1}(1+C_Me^{\overline\beta})\}$ times the joint summary distance. In the other direction, the same decomposition gives
\begin{equation}
\begin{aligned}
 \lVert t_J(\beta,p)-t_J(\beta_0,p_0)\rVert_2
 &\le |e^{-\beta}-e^{-\beta_0}|
       +\lVert s_J(\beta,p)-s_J(\beta_0,p_0)\rVert_2\\
 &\le (1+C_M)|\beta-\beta_0|
       +\overline\beta\lVert\mathrm{Proj}_{2J}(p-p_0)\rVert_2.
\end{aligned}
\end{equation}
Choosing $C_{\mathrm{st}}$ to dominate these fixed constants proves both inequalities in \eqref{eq:decompound-joint-stability}, uniformly in $J$.
\end{proof}

\begin{lemma}[Folded continuous component]
\label{lem:decompound-folding}
For $\beta>0$ and $p\in\mathcal D\cap\mathcal L^\infty(\mathcal X)$, map each coordinate of a total to its fractional part in $[0,1)$. Under this map, the continuous component of $\mathrm Q_{\beta,p}$ has density
\begin{equation}
  r_{\beta,p}^\circ(x)=e^{-\beta}\sum_{r=1}^{\infty}\frac{\beta^r}{r!}\,
  p^{(*_{\mathbb T}r)}(x),
  \qquad x\in[0,1)^d,
  \label{eq:decompound-folded-density}
\end{equation}
where $*_{\mathbb T}$ denotes convolution on the unit torus. This subprobability density satisfies
\begin{equation}
  \int_{[0,1)^d}r_{\beta,p}^\circ(x)\,\mathrm dx=1-e^{-\beta},
  \qquad
  \lVert r_{\beta,p}^\circ\rVert_\infty\le(1-e^{-\beta})\lVert p\rVert_\infty.
  \label{eq:decompound-folded-envelope}
\end{equation}
\end{lemma}

\begin{proof}
Conditionally on $C_i=r\ge1$, the total reduced modulo $\mathbb Z^d$ has density $p^{(*_{\mathbb T}r)}$; this follows by integrating the product jump density and taking the sum modulo $\mathbb Z^d$. Summing the nonnegative conditional densities over the Poisson probabilities proves \eqref{eq:decompound-folded-density} by Tonelli's theorem. Equivalently, $r_{\beta,p}^\circ(x)=\sum_{a\in\mathbb Z^d}q_{\beta,p}(x+a)$ almost everywhere. Convolution with a probability density on the torus does not increase the supremum norm, because
\begin{equation}
  \left|\int_{[0,1)^d} f(x-u)h(u)\,\mathrm du\right|
  \le\lVert f\rVert_\infty\int_{[0,1)^d}h(u)\,\mathrm du
  =\lVert f\rVert_\infty
\end{equation}
for a density $h$, with $f$ interpreted periodically. Every convolution in \eqref{eq:decompound-folded-density} integrates to one and has supremum norm at most $\lVert p\rVert_\infty$. Summing the weights gives \eqref{eq:decompound-folded-envelope}. In particular, for every bounded periodic function $h$,
\begin{equation}
  \int_{\mathbb R^d}\mathbf 1\{x\ne0\}h(x)^2\mathrm Q_{\beta,p}(\mathrm dx)
  =\int_{[0,1)^d}h(x)^2r_{\beta,p}^\circ(x)\,\mathrm dx.
  \label{eq:decompound-periodic-second}
\end{equation}
The atom at zero is excluded by the indicator, and the continuous component assigns zero mass to the point zero.
\end{proof}

For the uniform jump density $p=1$, $r_{\beta,p}^\circ=1-e^{-\beta}$. Put $K=D_{2J}-1$ and $u=\psi_{2J}(0)/\sqrt K$. Then the unmodified periodic statistic satisfies
\begin{equation}
  \operatorname{Var}_{\mathrm Q_{\beta,1}}\{u^\top\psi_{2J}(X_1)\}
  =e^{-\beta}(1-e^{-\beta})K+(1-e^{-\beta}).
\end{equation}
Indeed, its value at the atom is $\sqrt K$, and its first two moment integrals against the continuous component are zero and $1-e^{-\beta}$, respectively. This accounts for the zero-count correction in \eqref{eq:decompound-summary-center}.

\begin{proof}[Proof of Proposition~\ref{prop:density-concentration}]
Let $J=J_n$, $K=K_n$, and write a unit vector in $\mathbb R^{K+1}$ as $u=(a,v)$, where $a\in\mathbb R$, $v\in\mathbb R^K$, and $a^2+\lVert v\rVert_2^2=1$. The function $h_v=v^\top\psi_{2J}$ satisfies $\lVert h_v\rVert_2=\lVert v\rVert_2$ and $\lVert h_v\rVert_\infty\le\sqrt K\lVert v\rVert_2$. The two terms in
\begin{equation}
  u^\top\xi_J(X_1)=a\mathbf 1\{X_1=0\}
                  +\mathbf 1\{X_1\ne0\}h_v(X_1)
\end{equation}
have disjoint supports. Lemma~\ref{lem:decompound-folding} therefore gives, for bounded $p\in\mathcal D$,
\begin{equation}
\begin{aligned}
  \operatorname{Var}_{\mathrm Q_{\beta,p}}\{u^\top\xi_J(X_1)\}
  &\le a^2e^{-\beta}
      +\int_{[0,1)^d}h_v(x)^2r_{\beta,p}^\circ(x)\,\mathrm dx\\
  &\le a^2+\lVert p\rVert_\infty\lVert v\rVert_2^2
  \le\max\{1,\lVert p\rVert_\infty\}.
\end{aligned}
\label{eq:decompound-directional-variance}
\end{equation}
The identity $\sum_{j=2}^{D_{2J}}\varphi_j(x)^2=K$ also gives $\lVert\xi_J(x)\rVert_2\le\sqrt{K+1}$, and hence
\begin{equation}
 |u^\top\{\xi_J(X_1)-t_J(\beta,p)\}|\le2\sqrt{K+1}.
\end{equation}
These bounds are uniform over $\beta\in \mathcal I$ and do not assume independence between summary coordinates.

For the observed sample, put $W_i=\xi_J(X_i)-t_J(\beta_0,p_0)$. Apply Lemma~\ref{lem:bernstein-net} in dimension $K+1$, with sample size $n$, $V=M$, $H=2$, and $t=B\sqrt{K/n}$. Since $\sqrt{K(K+1)}\le\sqrt2K$,
\begin{equation}
\begin{aligned}
 &\mathrm P_{\beta_0,p_0}^{(n)}
 \left\{\lVert S_n(X)-t_J(\beta_0,p_0)\rVert_2>B\sqrt{K/n}\right\}\\
 &\qquad\le2\cdot5^{K+1}
 \exp\left\{-\frac{cB^2K}{M+2\sqrt2BK/\sqrt n}\right\}.
\end{aligned}
\end{equation}
Because $K\asymp n^{d/(2\alpha+d)}$ and $\alpha>d/2$, $K/\sqrt n\to0$. For each fixed $B$, the denominator is eventually at most $M+1$. The inequality $\log2+(K+1)\log5\le K\log50$ for $K\ge1$ shows that a sufficiently large fixed $B_0$ absorbs the prefactor into the exponential, giving \eqref{eq:density-observed-concentration} with a positive $c_{\mathrm{coef}}$ independent of $B,\beta_0,p_0$.

For synthetic data, $p\in\operatorname{supp}(\Pi_n)$ satisfies $p\le D_J$, and $D_J\ge1$. The centered vectors $\widetilde W_i=\xi_J(Z_i)-t_J(\beta,p)$ have the same almost-sure bound and directional variance at most $D_J$ by \eqref{eq:decompound-directional-variance}. Since $m_n\ge nD_J$, the same lemma gives
\begin{equation}
\begin{aligned}
 &\mathrm P_{\beta,p}^{(m_n)}
 \left\{\lVert S_{m_n}(Z)-t_J(\beta,p)\rVert_2>B\sqrt{K/n}\right\}\\
 &\qquad\le2\cdot5^{K+1}
 \exp\left\{-\frac{cB^2K}{1+2\sqrt2BK/(D_J\sqrt n)}\right\}.
\end{aligned}
\end{equation}
Here $K/D_J$ is bounded, so for every fixed $B$ the denominator is eventually at most two. Increasing $B_0$ and decreasing $c_{\mathrm{coef}}$ if necessary proves \eqref{eq:density-synthetic-concentration}, uniformly over $\beta\in \mathcal I$ and $p\in\operatorname{supp}(\Pi_n)$. The lower bound on $n$ in both inequalities may depend on $B$.
\end{proof}

\begin{proof}[Proof of Corollary~\ref{cor:density-synthetic-small-ball}]
Apply \eqref{eq:density-synthetic-concentration} with $B=B_0$. Since $K_n\to\infty$, the complementary probability is at most $1/2$ for all sufficiently large $n$, uniformly over $\beta\in \mathcal I$ and $p\in\operatorname{supp}(\Pi_n)$. Thus one may take $c_1=1/2$.
\end{proof}

\begin{proof}[Proof of Theorem~\ref{thm:density-ideal}]
We apply Theorem~\ref{thm:meta-ideal} to the joint parameter space $\Theta=\mathcal I\times\mathcal D$ with metric $d_\Theta$ in \eqref{eq:decompound-product-metric}, prior $\widetilde\Pi_n$, and target class $\mathcal I\times\mathcal P_\alpha(R,m,M)$. Let $J=J_n$, $K=K_n$, and choose a fixed $B_X\ge B_0$ from Proposition~\ref{prop:density-concentration}. Define
\begin{equation}
  \mathcal E_n=
  \left\{\lVert S_n(X)-t_J(\beta_0,p_0)\rVert_2\le B_X\sqrt{K/n}\right\}.
\end{equation}
Proposition~\ref{prop:density-concentration} gives (B1) of Condition~\ref{ass:meta}.

Let $\mathcal B_{n,p}(p_0)$ be the image of the closed coefficient ball in Proposition~\ref{prop:density-local-mass}. It is a compact Borel set with $\Pi_n\{\mathcal B_{n,p}(p_0)\}\ge e^{-C_0K}$ and $\lVert p-p_0\rVert_2\le C_*\epsilon_n$ throughout the set. Define the joint local set by
\begin{equation}
  \mathcal B_n(\beta_0,p_0)
  =\left(\mathcal I\cap[\beta_0-\epsilon_n,\beta_0+\epsilon_n]\right)
       \times\mathcal B_{n,p}(p_0).
\end{equation}
For $\epsilon_n\le\overline\beta-\underline\beta$, the intensity interval in this product has length at least $\epsilon_n$, including when $\beta_0$ is an endpoint of $\mathcal I$. Independence in the joint prior gives
\begin{equation}
  \widetilde\Pi_n\{\mathcal B_n(\beta_0,p_0)\}
  \ge w_-\epsilon_n e^{-C_0K}\ge e^{-C_\Pi K}
  \label{eq:decompound-joint-local-mass}
\end{equation}
for a fixed $C_\Pi$ and all sufficiently large $n$, because $\log(1/\epsilon_n)=O(\log n)=o(K)$. This proves (B2). The product-metric distance throughout this set is at most $(1+C_*)\epsilon_n$. Proposition~\ref{prop:decompound-joint-stability} gives
\begin{equation}
  \lVert t_J(\beta,p)-t_J(\beta_0,p_0)\rVert_2
  \le C_{\mathrm{st}}(1+C_*)\epsilon_n.
\end{equation}
Every density $p_v$ belongs to $\operatorname{supp}(\Pi_n)$: each open neighbourhood of $p_v$ has a nonempty open preimage under the continuous map $v\mapsto p_v$, and every conditional Gaussian coefficient law, hence also $\Gamma_n$, assigns positive probability to that preimage. On $\mathcal E_n$, the synthetic event
\begin{equation}
  \lVert S_{m_n}(Z)-t_J(\beta,p)\rVert_2\le B_0\sqrt{K/n}
\end{equation}
implies acceptance whenever $A>B_X+C_{\mathrm{st}}(1+C_*)+B_0$. Corollary~\ref{cor:density-synthetic-small-ball} bounds its probability below by $c_1>0$. Taking $C_L$ large enough for the local distance and this acceptance bound proves (B3).

Set $\mathcal G_n=\mathcal I\times\operatorname{supp}(\Pi_n)$, so (B5) holds. Every density component of $\mathcal G_n$ lies in $\mathcal V_{2J}$, whereas the target class gives
\begin{equation}
  \lVert p_0-\mathrm{Proj}_{2J}p_0\rVert_2
  \le R(1+2J)^{-\alpha}\le R\epsilon_n.
\end{equation}
For $(\beta,p)\in\mathcal G_n$, the triangle inequality and Proposition~\ref{prop:decompound-joint-stability} consequently imply
\begin{equation}
\begin{aligned}
 d_\Theta\left((\beta,p),(\beta_0,p_0)\right)
 &\le |\beta-\beta_0|+\lVert p-\mathrm{Proj}_{2J}p_0\rVert_2+R\epsilon_n\\
 &\le C_{\mathrm{st}}\lVert t_J(\beta,p)-t_J(\beta_0,p_0)\rVert_2+R\epsilon_n.
\end{aligned}
\end{equation}
If this parameter distance exceeds $L\epsilon_n$, acceptance on $\mathcal E_n$ forces
\begin{equation}
  \lVert S_{m_n}(Z)-t_J(\beta,p)\rVert_2
  >\left\{C_{\mathrm{st}}^{-1}(L-R)-B_X-A\right\}\epsilon_n.
\end{equation}
For every fixed sufficiently large $L$, apply \eqref{eq:density-synthetic-concentration} with $B=C_{\mathrm{st}}^{-1}(L-R)-B_X-A\ge B_0$, using $\epsilon_n\ge\sqrt{K/n}$. This is (B4), with
\begin{equation}
  C_D=c_{\mathrm{coef}}C_{\mathrm{st}}^{-2},
  \qquad C_S=R+C_{\mathrm{st}}(B_X+A).
\end{equation}
All constants and the required sample-size thresholds are uniform over $(\beta_0,p_0)\in \mathcal I\times\mathcal P_\alpha(R,m,M)$. Theorem~\ref{thm:meta-ideal} gives the joint posterior bound in the product metric on $\mathcal E_n$. Since $\lVert p-p_0\rVert_2\le d_\Theta((\beta,p),(\beta_0,p_0))$, the same bound applies to the density marginal outside $L\epsilon_n$. Finally, $\epsilon_n\asymp\rho_{n,\alpha,d}$ converts the radius after increasing the fixed threshold $L_0$, proving the stated marginal contraction.
\end{proof}

\begin{proof}[Proof of Proposition~\ref{prop:density-second-moment}]
Let $J=J_n$ and $K=K_n$. Fix a sufficiently large constant $L$ at the $\epsilon_n$ scale. By \eqref{eq:meta-ideal-on-event} and the proof of Theorem~\ref{thm:density-ideal}, the density-marginal posterior tail on $\mathcal E_n$ is bounded by $2e^{-\gamma_L K}$ for some $\gamma_L>0$ and all sufficiently large $n$, uniformly over $\beta_0\in \mathcal I$ and $p_0\in\mathcal P_\alpha(R,m,M)$. The support property following \eqref{eq:density-square-root-prior} implies
\begin{equation}
  \lVert p-p_0\rVert_2^2\le2\lVert p\rVert_2^2+2\lVert p_0\rVert_2^2
  \le2D_J+2M\le CK
\end{equation}
for every $p\in\operatorname{supp}(\Pi_n)$; the fallback density $p^\dagger=1$ satisfies the same bound. With positive normalizer the density marginal is absolutely continuous with respect to $\Pi_n$. Hence its posterior second moment is at most $L^2\epsilon_n^2+2CKe^{-\gamma_L K}$ on $\mathcal E_n$ and at most $CK$ on its complement. Taking expectation under $\mathrm P_{\beta_0,p_0}^{(n)}$, using the uniform bound on $\mathrm P_{\beta_0,p_0}^{(n)}(\mathcal E_n^c)$ and $Ke^{-cK}=O(\epsilon_n^2)$, proves \eqref{eq:density-second}.
\end{proof}

\begin{proof}[Proof of Corollary~\ref{cor:density-mean-risk}]
Proposition~\ref{prop:density-second-moment} gives Bochner integrability. Jensen's inequality in $\mathcal L^2$ proves the risk bound. The set $\mathcal D$ is closed and convex in $\mathcal L^2$, so the Bochner mean of a probability measure supported on $\mathcal D$ also belongs to $\mathcal D$.
\end{proof}

\begin{lemma}[Kullback--Leibler bound under compounding]
\label{lem:decompound-kl}
For jump densities $p,\widetilde p\in\mathcal D$ and the same fixed count intensity $\beta>0$,
\begin{equation}
  \mathrm{KL}(\mathrm Q_{\beta,p},\mathrm Q_{\beta,\widetilde p})\le\beta\,\mathrm{KL}(p,\widetilde p).
  \label{eq:decompound-kl}
\end{equation}
\end{lemma}

\begin{proof}
This is the fixed-intensity case of \citet[Lemma 1, inequality (11)]{GugushviliVanderMeulenSpreij2015}. It can also be seen directly by retaining the latent count and jumps. Let $\overline{\mathrm Q}_{\beta,p}$ be their joint law on the disjoint union $\bigcup_{r\ge0}\{r\}\times\mathcal X^r$, where $\mathcal X^0$ is a singleton. When $\mathrm{KL}(p,\widetilde p)<\infty$, conditioning on the count gives
\begin{equation}
  \mathrm{KL}(\overline{\mathrm Q}_{\beta,p},\overline{\mathrm Q}_{\beta,\widetilde p})
  =\sum_{r=0}^{\infty}e^{-\beta}\frac{\beta^r}{r!}\,
  r\,\mathrm{KL}(p,\widetilde p)
  =\beta\,\mathrm{KL}(p,\widetilde p).
\end{equation}
The total is a measurable function of this latent vector. Conditional Jensen's inequality for the convex function $x\log x$ shows that taking its distribution cannot increase Kullback--Leibler divergence, proving \eqref{eq:decompound-kl}. If $\mathrm{KL}(p,\widetilde p)=\infty$, the asserted upper bound holds in the extended sense. No smoothness or strict positivity beyond the absolute continuity implicit in a finite Kullback--Leibler divergence is needed.
\end{proof}

\begin{proof}[Proof of Theorem~\ref{thm:density-lower}]
For background on Sobolev density classes and minimax lower bounds, see \citet[Chapters 10 and 11]{Klemela2009}. Fix $\beta_\star\in \mathcal I$ and restrict to the submodel with intensity $\beta_\star$. We use a trigonometric hypercube within this submodel. For a large dyadic frequency level $J$, let $\mathcal K_J=\{2^J,\ldots,2^{J+1}-1\}^d$ index the cosine-product functions
\begin{equation}
  \zeta_{J,\boldsymbol q}(x)=\prod_{r=1}^d\sqrt2\cos(2\pi q_rx_r),
  \qquad 2^J\le q_r<2^{J+1}.
\end{equation}
They are orthonormal, have integral zero, have uniformly bounded supremum norm, and $|\mathcal K_J|=2^{Jd}$. For $\nu\in\{-1,1\}^{\mathcal K_J}$, set
\begin{equation}
  p_\nu=1+a_J\sum_{k\in\mathcal K_J}\nu_k\zeta_{J,k},
  \qquad
  a_J=c_a2^{-J(\alpha+d/2)}.
\end{equation}
For every selected basis function the block index satisfies $1+q_j\le2^{J+1}$. Thus
\begin{equation}
\begin{aligned}
  \sum_{j=2}^{\infty}(1+q_j)^{2\alpha}\theta_j(p_\nu)^2
  &\le 2^{2\alpha(J+1)}|\mathcal K_J|a_J^2
   =2^{2\alpha}c_a^2,\\
  \lVert p_\nu-1\rVert_\infty
  &\le 2^{d/2}|\mathcal K_J|a_J
   =2^{d/2}c_a2^{-J(\alpha-d/2)}.
\end{aligned}
\end{equation}
Choose $c_a>0$ so that $2^{2\alpha}c_a^2\le R^2$. Since $\alpha>d/2$ and $m<1<M$, the last bound is at most $\min\{1-m,M-1\}$ for all sufficiently large $J$, uniformly in $\nu$. Every $p_\nu$ then lies between $m$ and $M$ and integrates to one, so the whole cube is contained in $\mathcal P_\alpha(R,m,M)$.

For all vertices, orthonormality gives
\begin{equation}
  \lVert p_\nu-p_{\nu'}\rVert_2^2=4a_J^2H(\nu,\nu').
\end{equation}
In particular, neighboring densities differ by $2a_J\zeta_{J,k}$. Since every cube density is bounded below by $m$,
\begin{equation}
  \mathrm{KL}(p_\nu,p_{\nu^{(k)}})
  \le\int_{\mathcal X}\frac{(p_\nu(x)-p_{\nu^{(k)}}(x))^2}{p_{\nu^{(k)}}(x)}\,\mathrm dx
  \le 4a_J^2/m.
\end{equation}
For the full aggregate-observation model, Lemma~\ref{lem:decompound-kl} and independence across periods give
\begin{equation}
\begin{aligned}
  \mathrm{KL}(\mathrm P_{\beta_\star,p_\nu}^{(n)},\mathrm P_{\beta_\star,p_{\nu^{(k)}}}^{(n)})
  &=n\,\mathrm{KL}(\mathrm Q_{\beta_\star,p_\nu},\mathrm Q_{\beta_\star,p_{\nu^{(k)}}})\\
  &\le n\beta_\star\,\mathrm{KL}(p_\nu,p_{\nu^{(k)}})
  \le 4n\beta_\star a_J^2/m.
\end{aligned}
\end{equation}
Take
\begin{equation}
  J=\left\lfloor\frac{\log_2 n}{2\alpha+d}\right\rfloor,
  \qquad
  2^{J(2\alpha+d)}\le n<2^{(J+1)(2\alpha+d)}.
\end{equation}
Then $J\to\infty$ and $na_J^2\le2^{2\alpha+d}c_a^2$. Pinsker's inequality and the preceding Kullback--Leibler bound give
\begin{equation}
\begin{aligned}
  &\mathrm{TV}(\mathrm P_{\beta_\star,p_\nu}^{(n)},
               \mathrm P_{\beta_\star,p_{\nu^{(k)}}}^{(n)})\\
  &\quad\le\left\{\frac12\mathrm{KL}(\mathrm P_{\beta_\star,p_\nu}^{(n)},
                    \mathrm P_{\beta_\star,p_{\nu^{(k)}}}^{(n)})\right\}^{1/2}\\
  &\quad\le\left(\frac{2n\beta_\star a_J^2}{m}\right)^{1/2}
   \le c_a\left(\frac{2^{2\alpha+d+1}\beta_\star}{m}\right)^{1/2}.
\end{aligned}
\end{equation}
Reduce $c_a$, if necessary, so that $c_a^2\le m/(2^{2\alpha+d+3}\beta_\star)$. Every neighboring total variation distance is then at most $1/2$, and the target-class inclusion is preserved.

Theorem~\ref{thm:external-assouad} applies to the fixed-intensity hypercube with dimension $|\mathcal K_J|$, squared separation parameter $4a_J^2$, and $\eta=1/2$. Keeping the infimum over measurable estimators based on the full aggregate sample throughout gives
\begin{equation}
\begin{aligned}
  &\inf_{\widetilde p_n}\sup_{\substack{\beta_0\in\mathcal I\\p_0\in\mathcal P_\alpha(R,m,M)}}
    \mathrm E_{\beta_0,p_0}^{(n)}\lVert\widetilde p_n-p_0\rVert_2^2\\
  &\quad\ge\inf_{\widetilde p_n}\sup_{p_0\in\mathcal P_\alpha(R,m,M)}
    \mathrm E_{\beta_\star,p_0}^{(n)}\lVert\widetilde p_n-p_0\rVert_2^2\\
  &\quad\ge\inf_{\widetilde p_n}\max_{\nu\in\{-1,1\}^{\mathcal K_J}}
    \mathrm E_{\beta_\star,p_\nu}^{(n)}\lVert\widetilde p_n-p_\nu\rVert_2^2\\
  &\quad\ge\frac{|\mathcal K_J|(4a_J^2)}8\left(1-\frac12\right)
   =\frac{c_a^2}{4}2^{-2\alpha J}\\
  &\quad\ge\frac{c_a^2}{4}n^{-2\alpha/(2\alpha+d)}.
\end{aligned}
\end{equation}
The last inequality uses $2^J\le n^{1/(2\alpha+d)}$. The fixed value $\beta_\star$ specifies the submodel used in the lower bound; it is not an additional observation supplied to the estimator. The constant $c_a^2/4$ is positive and independent of $n$, proving the theorem.
\end{proof}

\begin{proof}[Proof of Theorem~\ref{thm:density-mc}]
The proof of Theorem~\ref{thm:density-ideal} verifies Condition~\ref{ass:meta} on the joint parameter space. Corollary~\ref{cor:meta-denom} therefore gives constants $\kappa,c'>0$ for the joint acceptance normalizer, uniformly over $\beta_0\in \mathcal I$ and $p_0\in\mathcal P_\alpha(R,m,M)$. For a Borel density set $B$, apply the conditional accepted-draw argument in \eqref{eq:proof-mc-mass} to $\mathcal I\times B$ and then take the density marginal:
\begin{equation}
 \mathrm E_{\mathrm{MC}}\{\widehat\Pi_{n,N}^{\mathrm{ABC}}(B\mid X)\mid X\}
 \le\Pi_n^{\mathrm{ABC}}(B\mid X)+(1-D_n(X))^{N_n}.
\end{equation}
The expected no-acceptance term is bounded uniformly by
\begin{equation}
 e^{-c'K_n}+\exp\left\{-N_ne^{-\kappa K_n}\right\}.
 \label{eq:decompound-no-acceptance}
\end{equation}
For $B=\{p:\lVert p-p_0\rVert_2>L\rho_{n,\alpha,d}\}$, Theorem~\ref{thm:density-ideal} shows that the sampling expectation of $\Pi_n^{\mathrm{ABC}}(B\mid X)$ tends to zero uniformly over $\beta_0\in \mathcal I$ and $p_0\in\mathcal P_\alpha(R,m,M)$. Both terms in \eqref{eq:decompound-no-acceptance} vanish under the stated budget, proving the claim.
\end{proof}

\begin{proof}[Proof of Corollary~\ref{cor:density-mc-risk}]
Conditional on $X$ and $A_{n,N}=r\ge1$, the accepted pairs are iid from $\widetilde\Pi_n^{\mathrm{ABC}}(\cdot\mid X)$. Their density components $F_1,\ldots,F_r$ are therefore iid from the marginal $\Pi_n^{\mathrm{ABC}}(\cdot\mid X)$. Jensen's inequality in $\mathcal L^2$ gives
\begin{equation}
\begin{aligned}
 &\mathrm E_{\mathrm{MC}}\left\{\left.
   \left\lVert r^{-1}\sum_{j=1}^r F_j-p_0\right\rVert_2^2
   \,\right|X,A_{n,N}=r\right\}\\
 &\qquad\le\int_{\mathcal D}\lVert p-p_0\rVert_2^2
              \Pi_n^{\mathrm{ABC}}(\mathrm dp\mid X).
\end{aligned}
\end{equation}
Proposition~\ref{prop:density-second-moment} bounds the sampling expectation of the right-hand side uniformly by $C\rho_{n,\alpha,d}^2$. On $A_{n,N}=0$, the empirical density mean is one, with squared loss
\begin{equation}
  \lVert1-p_0\rVert_2^2=\lVert p_0\rVert_2^2-1\le M-1.
\end{equation}
Average over the acceptance count and the observed data. The resulting risk is at most
\begin{equation}
 C\rho_{n,\alpha,d}^2
 +(M-1)\left(e^{-c'K_n}+\exp\left\{-N_ne^{-\kappa K_n}\right\}\right).
\end{equation}
Polynomial divergence of $K_n$ makes the first exponential $O(\rho_{n,\alpha,d}^2)$. Under the stated budget the second is at most $\rho_{n,\alpha,d}^{C_N}$, which is $O(\rho_{n,\alpha,d}^2)$ for sufficiently large $C_N$. This proves the density-risk bound.
\end{proof}

\section{Parametric minimax optimality results}
\label{app:parametric}

The following fixed-dimensional minimax results complement the parametric ABC analyses of \citet{FrazierMartinRobertRousseau2018} and \citet{LiFearnhead2018a,LiFearnhead2018b}.

\subsection{Parametric setup}

Let $d_\theta\in\mathbb N$, let $\Theta\subset\mathbb{R}^{d_\theta}$ be compact, and let $\Theta_0\subset\mathrm{int}(\Theta)$ have positive distance from the boundary.  For each $n$ and $\theta\in\Theta$, let $\mathrm{P}_\theta^{(n)}$ be the data law, and let $S_n(X^{(n)})\in\mathbb{R}^{d_\theta}$ be the ABC summary statistic.  The prior has density $\pi$ satisfying $0<\pi_{\min}\le\pi(\theta)\le\pi_{\max}<\infty$ on $\Theta$.  Fix a parametric fallback point $\theta_{\mathrm{par}}^\dagger\in\Theta$.

\begin{condition}[Uniform root-$n$ concentration and small-ball condition]
\label{ass:param-rootn}
There exist $C_S,c_S>0$ such that uniformly over $\theta\in\Theta$,
\begin{equation}
  \mathrm{P}_\theta^{(n)}\left\{\sqrt n\left\lVert S_n(X^{(n)})-\theta \right\rVert_2>t\right\}\le C_S e^{-c_S t^2},\qquad t\ge0,
  \label{eq:param-root-tail}
\end{equation}
\begin{equation}
  \mathrm{E}_\theta^{(n)}\left\lVert S_n(X^{(n)})-\theta \right\rVert_2^2\le C_S/n,
  \label{eq:param-root-second}
\end{equation}
and for every fixed $a>0$ there is $c_{\mathrm{sb}}(a)>0$ such that
\begin{equation}
  \inf_{\theta\in\Theta}\mathrm{P}_\theta^{(n)}\left\{\sqrt n\left\lVert S_n(X^{(n)})-\theta \right\rVert_2\le a\right\}\ge c_{\mathrm{sb}}(a)
  \label{eq:param-root-small}
\end{equation}
for all sufficiently large $n$.
\end{condition}

Set $\tau_n=A/\sqrt n$.  The ideal parametric ABC likelihood and posterior are
\begin{equation}
  L_{n,\mathrm{par}}^{\mathrm{ABC}}(\theta;X)=\mathrm{P}_\theta^{(n)}\left\{\left\lVert S_n(Z^{(n)})-S_n(X^{(n)}) \right\rVert_2\le A/\sqrt n\right\},
\end{equation}
\begin{equation}
  \Pi_{n,\mathrm{par}}^{\mathrm{ABC}}(B\mid X)=\frac{\int_B L_{n,\mathrm{par}}^{\mathrm{ABC}}(\theta;X)\pi(\theta)\mathrm d\theta}{\int_\Theta L_{n,\mathrm{par}}^{\mathrm{ABC}}(\theta;X)\pi(\theta)\mathrm d\theta}.
\end{equation}
Write the denominator of this posterior as
\begin{equation}
  D_{n,\mathrm{par}}(X)=\int_\Theta L_{n,\mathrm{par}}^{\mathrm{ABC}}(\theta;X)\pi(\theta)\mathrm d\theta.
\end{equation}
When $D_{n,\mathrm{par}}(X)=0$, set $\Pi_{n,\mathrm{par}}^{\mathrm{ABC}}(\cdot\mid X)=\delta_{\theta_{\mathrm{par}}^\dagger}$.  The corresponding rejection sampler with $N_n$ prior samples defines the empirical law $\widehat\Pi_{n,N,\mathrm{par}}^{\mathrm{ABC}}$ by the same construction as \eqref{eq:empirical-posterior-general}, using $\delta_{\theta_{\mathrm{par}}^\dagger}$ on the no-acceptance event and replacing $L_n^{\mathrm{ABC}}$ by $L_{n,\mathrm{par}}^{\mathrm{ABC}}$.

\subsection{Upper bound}

\begin{theorem}[Parametric ABC posterior second-moment upper bound]
\label{thm:param-upper}
Under Condition~\ref{ass:param-rootn}, for every fixed $A>0$,
\begin{equation}
  \sup_{\theta_0\in\Theta_0}\mathrm{E}_{\theta_0}^{(n)}\int_\Theta \left\lVert \theta-\theta_0 \right\rVert_2^2\Pi_{n,\mathrm{par}}^{\mathrm{ABC}}(\mathrm d\theta\mid X)
  \le C/n.
  \label{eq:param-upper-second}
\end{equation}
\end{theorem}

\begin{proof}
Let $s=S_n(X^{(n)})$ and let $r_0=\mathrm{dist}(\Theta_0,\Theta^c)>0$.  On the event $\left\lVert s-\theta_0 \right\rVert_2\le r_0/2$, the ball $B(s,A/(2\sqrt n))$ is contained in $\Theta$ for all large $n$.  For every $\theta$ in this ball, the small-ball condition gives
\begin{equation}
  L_{n,\mathrm{par}}^{\mathrm{ABC}}(\theta;X)\ge c_{\mathrm{sb}}(A/2).
\end{equation}
Thus the ABC denominator is at least $C_D n^{-d_\theta/2}$ on this event.  The event fails with probability at most $C_S e^{-c_S r_0^2 n/4}$ by \eqref{eq:param-root-tail}.

For the numerator around $s$, acceptance implies
\begin{equation}
  \sqrt n\left\lVert S_n(Z^{(n)})-\theta\right\rVert_2
  \ge (\sqrt n\lVert\theta-s\rVert_2-A)_+.
\end{equation}
If the right-hand side is positive, the event with a weak inequality is contained in the event in \eqref{eq:param-root-tail} with half this threshold; when it is zero, the likelihood is at most one. Thus, after changing constants, uniformly in $\theta$,
\begin{equation}
  L_{n,\mathrm{par}}^{\mathrm{ABC}}(\theta;X)
  \le C\exp\{-c(\sqrt n\left\lVert \theta-s \right\rVert_2-A)_+^2\}.
\end{equation}
Therefore, using $u=\sqrt n(\theta-s)$,
\begin{align}
  \int_\Theta \left\lVert \theta-s \right\rVert_2^2 L_{n,\mathrm{par}}^{\mathrm{ABC}}(\theta;X)\pi(\theta)\mathrm d\theta
  &\le C n^{-d_\theta/2-1}\int_{\mathbb R^{d_\theta}}\left\lVert u \right\rVert_2^2 e^{-c(\left\lVert u \right\rVert_2-A)_+^2}\mathrm du\\
  &\le C'n^{-d_\theta/2-1}.
\end{align}
Dividing by the denominator gives
\begin{equation}
  \int_\Theta \left\lVert \theta-s \right\rVert_2^2\Pi_{n,\mathrm{par}}^{\mathrm{ABC}}(\mathrm d\theta\mid X)\le C/n
\end{equation}
on the high-probability event.  On that event,
\begin{equation}
  \int_\Theta\left\lVert \theta-\theta_0 \right\rVert_2^2\Pi_{n,\mathrm{par}}^{\mathrm{ABC}}(\mathrm d\theta\mid X)
  \le \frac{2C}{n}+2\left\lVert s-\theta_0 \right\rVert_2^2.
\end{equation}
On the complementary event, both $\theta$ and $\theta_0$ belong to the compact set $\Theta$, so the posterior integral of $\left\lVert\theta-\theta_0\right\rVert_2^2$ is bounded by the squared diameter of $\Theta$.  Taking expectations, using \eqref{eq:param-root-second}, and using the exponentially small probability of the complementary event proves \eqref{eq:param-upper-second}.
\end{proof}

\begin{corollary}[Parametric posterior mean risk]
\label{cor:param-mean-risk}
Under the conditions of Theorem~\ref{thm:param-upper}, the ideal ABC posterior mean $\widehat\theta_{n,\mathrm{par}}^{\mathrm{ABC}}=\int_\Theta\theta\Pi_{n,\mathrm{par}}^{\mathrm{ABC}}(\mathrm d\theta\mid X)$ satisfies
\begin{equation}
  \sup_{\theta_0\in\Theta_0}\mathrm{E}_{\theta_0}^{(n)}\left\lVert \widehat\theta_{n,\mathrm{par}}^{\mathrm{ABC}}-\theta_0 \right\rVert_2^2\le C/n.
\end{equation}
\end{corollary}

\begin{proof}
This follows from Jensen's inequality and Theorem~\ref{thm:param-upper}.
\end{proof}

\begin{corollary}[Parametric posterior tightness]
\label{cor:param-tightness}
Under the conditions of Theorem~\ref{thm:param-upper}, for every $M_n\to\infty$,
\begin{equation}
  \sup_{\theta_0\in\Theta_0}\mathrm{E}_{\theta_0}^{(n)}\Pi_{n,\mathrm{par}}^{\mathrm{ABC}}\left\{\theta:\left\lVert \theta-\theta_0 \right\rVert_2>M_n/\sqrt n\mid X\right\}\le C/M_n^2\to0.
\end{equation}
\end{corollary}

\begin{proof}
This follows from Markov's inequality under the posterior, followed by expectation over the observed data and Theorem~\ref{thm:param-upper}.
\end{proof}

\subsection{Lower bound and Monte Carlo estimation}

\begin{condition}[Local Kullback--Leibler regularity]
\label{ass:param-kl}
There are $\theta_\star\in\Theta_0$, $r_\star>0$, and $C_{\mathrm{KL}}<\infty$ such that $B(\theta_\star,r_\star)\subset\Theta_0$ and
\begin{equation}
  \mathrm{KL}(\mathrm{P}_\theta^{(n)},\mathrm{P}_\vartheta^{(n)})\le C_{\mathrm{KL}}n\left\lVert \theta-\vartheta \right\rVert_2^2
\end{equation}
for all $\theta,\vartheta\in B(\theta_\star,r_\star)$ and all sufficiently large $n$.
\end{condition}

\begin{theorem}[Parametric minimax lower bound]
\label{thm:param-lower}
Under Condition~\ref{ass:param-kl},
\begin{equation}
  \inf_{\widetilde\theta_n}\sup_{\theta_0\in\Theta_0}\mathrm{E}_{\theta_0}^{(n)}\left\lVert \widetilde\theta_n-\theta_0 \right\rVert_2^2\ge c/n
\end{equation}
for all sufficiently large $n$.
\end{theorem}

\begin{proof}
Let $e_1$ be a unit vector, set $\vartheta_0=\theta_\star$, and set $\vartheta_1=\theta_\star+h e_1/\sqrt n$, with $h>0$ small.  Then $\vartheta_1\in\Theta_0$ for large $n$ and $\mathrm{KL}(\mathrm{P}_{\vartheta_0}^{(n)},\mathrm{P}_{\vartheta_1}^{(n)})\le C_{\mathrm{KL}}h^2$.  By Pinsker's inequality, $\mathrm{TV}(\mathrm{P}_{\vartheta_0}^{(n)},\mathrm{P}_{\vartheta_1}^{(n)})\le \{\mathrm{KL}(\mathrm{P}_{\vartheta_0}^{(n)},\mathrm{P}_{\vartheta_1}^{(n)})/2\}^{1/2}$, so choose $h$ small enough that this total variation distance is at most $1/2$.  Theorem~\ref{thm:external-lecam} gives
\begin{equation}
  \inf_{\widetilde\theta_n}\sup_{j=0,1}\mathrm{E}_{\vartheta_j}^{(n)}\left\lVert \widetilde\theta_n-\vartheta_j \right\rVert_2^2
  \ge \frac{\left\lVert \vartheta_1-\vartheta_0 \right\rVert_2^2}{8}\left(1-\mathrm{TV}(\mathrm{P}_{\vartheta_0}^{(n)},\mathrm{P}_{\vartheta_1}^{(n)})\right)
  \ge c/n.
\end{equation}
\end{proof}

\begin{corollary}[Parametric minimax-rate optimality]
\label{cor:param-minimax}
Under Conditions~\ref{ass:param-rootn} and~\ref{ass:param-kl}, the ideal ABC posterior mean is minimax-rate optimal over $\Theta_0$ at squared-risk rate $n^{-1}$.
\end{corollary}

\begin{proof}
Corollary~\ref{cor:param-mean-risk} gives the matching upper bound, and Theorem~\ref{thm:param-lower} gives the minimax lower bound over all measurable estimators.
\end{proof}

\begin{theorem}[Parametric rejection-ABC posterior tightness]
\label{thm:param-mc}
Under Condition~\ref{ass:param-rootn}, if $N_n n^{-d_\theta/2}\to\infty$, then the empirical rejection-ABC posterior is root-$n$ tight in expectation: for every $M_n\to\infty$,
\begin{equation}
  \sup_{\theta_0\in\Theta_0}\mathrm{E}_{\theta_0}^{(n)}\mathrm{E}_{\mathrm{MC}}
  \widehat\Pi_{n,N,\mathrm{par}}^{\mathrm{ABC}}\left\{\theta:\left\lVert \theta-\theta_0 \right\rVert_2>M_n/\sqrt n\mid X\right\}\to0.
\end{equation}
\end{theorem}

\begin{proof}
Let $r_0=\mathrm{dist}(\Theta_0,\Theta^c)>0$.  For a measurable set $B$, the conditional argument in \eqref{eq:proof-mc-mass} gives
\begin{equation}
  \mathrm{E}_{\mathrm{MC}}\left\{\widehat\Pi_{n,N,\mathrm{par}}^{\mathrm{ABC}}(B\mid X)\mid X\right\}
  \le \Pi_{n,\mathrm{par}}^{\mathrm{ABC}}(B\mid X)+\mathrm{P}_{\mathrm{MC}}(A_{n,N}=0\mid X).
\end{equation}
The parametric denominator lower bound in the proof of Theorem~\ref{thm:param-upper} is $D_{n,\mathrm{par}}(X)\ge C_D n^{-d_\theta/2}$ on the event $\left\lVert S_n(X)-\theta_0 \right\rVert_2\le r_0/2$, whose complement has probability at most $C e^{-c n}$.  Hence
\begin{equation}
  \mathrm{E}_{\theta_0}^{(n)}\mathrm{P}_{\mathrm{MC}}(A_{n,N}=0\mid X)
  \le C e^{-c n}+\exp\{-C_D N_n n^{-d_\theta/2}\}.
\end{equation}
Take $B=\{\theta:\left\lVert\theta-\theta_0\right\rVert_2>M_n/\sqrt n\}$.  Corollary~\ref{cor:param-tightness} controls the ideal posterior term, and the displayed no-acceptance bound tends to zero when $N_n n^{-d_\theta/2}\to\infty$.
\end{proof}

\begin{corollary}[Parametric empirical posterior mean risk]
\label{cor:param-mc-mean}
Under Condition~\ref{ass:param-rootn}, suppose that $N_n n^{-d_\theta/2}\ge C_N\log n$, where $C_N$ is sufficiently large.  Then the empirical ABC posterior mean
\begin{equation}
  \widehat\theta_{n,N,\mathrm{par}}^{\mathrm{ABC}}
  =\int_\Theta\theta\,\widehat\Pi_{n,N,\mathrm{par}}^{\mathrm{ABC}}(\mathrm d\theta\mid X)
\end{equation}
satisfies
\begin{equation}
  \sup_{\theta_0\in\Theta_0}\mathrm{E}_{\theta_0}^{(n)}\mathrm{E}_{\mathrm{MC}}
  \left\lVert\widehat\theta_{n,N,\mathrm{par}}^{\mathrm{ABC}}-\theta_0\right\rVert_2^2
  \le C/n.
\end{equation}
\end{corollary}

\begin{proof}
Condition on $X$ and on $A_{n,N}=m\ge1$.  The accepted parameter values $F_1,\ldots,F_m$ are iid with law $\Pi_{n,\mathrm{par}}^{\mathrm{ABC}}(\cdot\mid X)$.  Jensen's inequality and conditional expectation therefore give
\begin{equation}
  \mathrm{E}_{\mathrm{MC}}\left\{\left\lVert m^{-1}\sum_{j=1}^mF_j-\theta_0\right\rVert_2^2\mid X,A_{n,N}=m\right\}
  \le \int_\Theta\left\lVert\theta-\theta_0\right\rVert_2^2\Pi_{n,\mathrm{par}}^{\mathrm{ABC}}(\mathrm d\theta\mid X).
\end{equation}
On the no-acceptance event the empirical posterior mean is $\theta_{\mathrm{par}}^\dagger$, and its loss is bounded by the squared diameter of $\Theta$.  Taking expectations, applying Theorem~\ref{thm:param-upper}, and using the no-acceptance bound from the proof of Theorem~\ref{thm:param-mc} proves the result.  The latter bound is $O(n^{-1})$ when $N_n n^{-d_\theta/2}\ge C_N\log n$ and $C_N$ is sufficiently large.
\end{proof}

\end{document}